\documentclass[a4paper]{cas-sc}

\usepackage[authoryear,longnamesfirst]{natbib}
\setcitestyle{comma,numbers,square}
\def\tsc#1{\csdef{#1}{\textsc{\lowercase{#1}}\xspace}}
\tsc{WGM}
\tsc{QE}
\tsc{EP}
\tsc{PMS}
\tsc{BEC}
\tsc{DE}

\begin{document}
\let\WriteBookmarks\relax
\def\floatpagepagefraction{1}
\def\textpagefraction{.001}
\setcitestyle{square,numbers}


\shorttitle{The Cheeger constant of curved tubes in real space forms}    

\shortauthors{Petr Vlachopulos} 

\title[mode = title]{The Cheeger constant of curved tubes in real space forms}                  

\author[1]{Petr Vlachopulos}[orcid=0009-0001-6410-7275]
\cormark[1]
\ead{pvlachopulos@gmail.com}
\cortext[1]{Corresponding author}
\ead[url]{https://www.researchgate.net/profile/Petr-Vlachopulos?ev=hdr_xprf}

\affiliation[1]{organization={Department of Mathematics and Statistics, Masaryk University, Faculty of Science},
                addressline={Kotlarska 2}, 
                city={Brno},
                postcode={611 37}, 
                country={Czech Republic}}

\begin{abstract}
Motivated by the geometric properties of curved tubes $T\left(P,a\right)$ defined by closed curves $P$, we compute the \textit{Cheeger constant} $h\left(T\left(P,a\right)\right)$ in the real space forms of arbitrary dimensions with constant sectional curvature. The structure and properties of the system of Fermi coordinates allows us to parametrize the \textit{curved tube} and straightforwardly compute the upper bound of $h\left(T\left(P,a\right)\right)$ using the exact formulas for the area and volume of $T\left(P,a\right)$. Next, we derive the lower bound by combining the geometry of tubes with a calibration-type argument. The key idea is to describe the tube using geodesic spheres moving along the underlying curve $P$, which provides a natural outward direction and makes the estimate geometrically transparent. This allows us to compute the lower bound via the divergence theorem. Finally, for the class of unbounded curved tubes in noncompact real space forms, we also compute the Cheeger constant and prove that there is no finite-volume Cheeger set.

\end{abstract}


\maketitle

\section{Introduction}

Let $M$ be a smooth, complete, orientable Riemannian manifold with a smooth boundary $\partial M$. We define the \textit{Cheeger constant} of $M$ as the number
\begin{equation}\label{eq:1}
\begin{aligned}
h\left(M\right)=\inf \left\{\frac{\left| \partial N\right| }{\left| N\right|},N\subset M\right\}, 
\end{aligned}
\end{equation}
where the infimum is taken over all compact smooth Riemannian submanifolds $N$ of $M$ with smooth boundary $\partial N$ and with the same dimension $n$. We denote by $\left| N\right|$ the $n$-dimensional Riemannian volume of $N$ and by $\left| \partial N\right|$ we denote the $(n-1)$-dimensional Riemannian area. Any submanifold $Q\subset M$ for which ${\frac{\left| \partial Q\right| }{\left| Q\right| }=h\left(M\right)}$ is called a \textit{Cheeger set} of $M$. In other words, any minimizer of (\ref{eq:1}) (if it exists) can be considered a \textit{Cheeger set} of $M$ and is denoted as $\mathscr{C}_{M}$. Hence, ${\frac{\left| \partial \mathscr{C}_{M}\right| }{\left| \mathscr{C}_{M}\right| }=h(M)}$. 
The \textit{Cheeger problem} consists of computation of the isoperimetric quotient $h(M)$ and the characterization of the \textit{Cheeger set} $\mathscr{C}_{M}$. If $\mathscr{C}_{M}=M$, we call $M$ \textit{self-Cheeger}. For more details about the history, fundamental formulations, and interesting properties of the Cheeger problem, the interested reader may see the following articles \cite{leonardi2015overview, maggi2012sets, singmaster1978constrained}. Note that the \textit{Cheeger constant} can also be defined for manifolds without boundary. See the original paper \cite{cheeger1970lower}.

There is a deep connection between the \textit{Cheeger problem} and the eigenvalue problem for the $p$-Laplacian when $p=2$, which is manifested through the following inequality 
\begin{equation*}
{\lambda _1(\Omega )\geq \frac{h^2(\Omega )}{4}}
\end{equation*}
where $\Omega \subset \mathbb{R}^N$ is a bounded domain, $\lambda _1(\Omega )$ represents the first eigenvalue of the $2$-Laplacian considered under Dirichlet boundary conditions. Such an inequality was first established in 1969 by American topologist Jeff Cheeger in his foundational paper \cite{cheeger1970lower}.

We also refer to the papers \cite{parini2011introduction, singmaster1978constrained, kawohl2003isoperimetric}, where the above-mentioned problem is studied in detail. One can see that the relation between the mentioned eigenvalue problem and the isoperimetric problem forms a backbone of current active research areas. 

Surprisingly, there are few situations where it is possible to compute the \textit{Cheeger constant} explicitly. One of the famous examples is the $n$-dimensional ball. Consider the $n$-dimensional Euclidean ball $B_a = {\left\{\left.x\in \mathbb{R}^n\right|\lVert x \lVert <a \right\}}$, where $a>0$. Then its \textit{Cheeger constant} is given as
\begin{equation*}
h(B_a)=\frac{n}{a}  
\end{equation*}
and its \textit{Cheeger set} is the $n$-ball itself, $C_{B_a}=B_a$.
Many results have been derived for the planar domains, i.e., $n=2$. For example, we can explicitly compute the \textit{Cheeger constant} of any convex polygon, see \cite{kawohl2006characterization}. In fact, for an arbitrary convex polygon, it is possible to derive a constructive algorithm for finding the \textit{Cheeger constant} and the \textit{Cheeger set} of such a polygon \cite{kawohl2006characterization}. Some recent results deal with the \textit{Cheeger sets} for rotationally symmetric planar convex bodies, where the \textit{Cheeger set} of such domains touches all the edges of the domain and is also rotationally symmetric \cite{canete2022cheeger}. Moreover, precise estimate for the \textit{Cheeger constant} and explicit characterization of the \textit{Cheeger set} for the so-called \textit{curved strips} were given in \cite{krejcirik2011cheeger}. Here, as in the case of the $n$-ball, the \textit{curved strips} are also \textit{self-Cheeger}. Remarkably, the \textit{Cheeger constant} of \textit{curved strips} corresponds to the \textit{Cheeger constant} of $1$-ball, independently of their shapes. A similar result resolving the \textit{Cheeger constants} of the \textit{curved tubes} in $\mathbb{R}^n$ was proved in \cite{krejcirik2019cheeger}. Despite the plethora of explicit results in low dimensions, much less is known about the \textit{Cheeger problem} in higher dimensions.

The aim of this paper is to provide an extension of our previous result \cite{krejcirik2019cheeger} to all space forms and to prove that the \textit{curved tubes} are the \textit{Cheeger sets} of themselves. We can define the \textit{curved tube} as a geometric shape wrapped around some central curve, while having a fixed radius. We adopt the following definition
\newdefinition{definition}{Definition}\label{def:1}
\begin{definition}
\textbf {(Curved tube)} \cite[pp.~32-33]{gray2003tubes} \textit{Let $M$ be a Riemannian manifold. Let $P$ be a smooth embedded closed curve. For an arbitrary $x\in M$, take the function $d_g(x) := \operatorname{dist}\left(x,P\right)$ as the geodesic distance between $x$ and $P$. Then the curved tube determined by $P$ is given by}
\begin{equation*}
T(P,a) := \{x\in M \vert \operatorname{dist}\left(x,P\right) \leq a\}.
\end{equation*}
\end{definition}
Our main result is stated below.

\newtheorem{thm}{Theorem}\label{thm:1}
\begin{thm}
\label{1}
\textit{Let $X$ be a locally symmetric oriented space with constant sectional curvature $K^X=const.$ Let $P$ be a smooth embedded closed curve in $X$ and $a>0$. Let $T(P,a)$ be the $n$-dimensional curved tube determined by $P$ in $X$. Assume that $a>0$ is small enough to ensure that $T(P,a)$ does not self-intersect. Then the \textit{Cheeger constant} of $T(P,a)$ equals}
\begin{equation*} 
      h\left(T\left(P,a\right)\right)=\begin{cases}
\vspace{0.25cm}
\left(n-1\right)\sqrt{K^X}\operatorname{cotg}\left(\sqrt{K^X}a\right), & K^X>0 \\
\vspace{0.25cm}
\hspace*{2.5cm}  \frac{n-1}{a}, & K^X=0 \\
\left(n-1\right)\sqrt{\vert K^X\vert}\operatorname{cotgh}\left(\sqrt{\vert K^X\vert}a\right), & K^X<0
    \end{cases}  
\end{equation*}
\textit{and $T\left(P,a\right)=\mathscr{C}_{T\left(P,a\right)}$}. 
\end{thm}

In the spherical case \(K^X>0\), the assumption that the radius \(a>0\) is sufficiently small for \(T\left(P,a\right)\) to form a regular embedded tubular neighborhood implies that \(a<\frac{\pi}{2\sqrt{K^X}}\). Consequently, the Cheeger constant remains positive \(\sqrt{K^X}\operatorname{cotg}\left(\sqrt{K^X}a\right)>0\).
Therefore, the positivity of the spherical expression follows automatically from the geometric admissibility of the tube and does not constitute an additional restriction on \(a\). Depending on the geometry of \(P\), the admissible radius may, in fact, be smaller. 

An important and interesting feature of this result is that the \textit{Cheeger constant} does not depend on the shape of the defining curve. If we interpret this geometric invariance in the context of the solution to the eigenvalue problem for $1$-Laplacian \cite{singmaster1978constrained}, then we have the \textit{Cheeger constant} as the first eigenvalue of the $1$-Laplacian. Thus, within this class of admissible curved tubes, the first eigenvalue is independent of the shape of the defining curve \(P\), although the geometry of the domain and the corresponding eigenfunctions may still depend on \(P\).

In order to prove theorem \ref{thm:1}, it was not possible to directly generalize the methods from the Euclidean case \cite{krejcirik2011cheeger}. Therefore, we had to develop new techniques.

Also, notice that the compactness of $P$ ensures that we can always find $a>0$ that is small enough, as in theorem \ref{thm:1}. The condition that $T(P,a)$ does not intersect itself is essential for our methods. We will analyze these technical issues in detail in the next section.

\newpage
\section{Preliminaries}

In the following paragraphs, all considerations will be situated in space forms with constant sectional curvature. Let $X$ be a smooth, complete Riemannian manifold with constant sectional curvature $K^X$. Thus $X$ corresponds to $\mathbb{R}^n$ with $K^X=0$, $\mathbb{S}^n$ with $K^X>0$, and the hyperbolic case $\mathbb{H}^{n}$ with $K^X<0$, up to isomorphisms. 

Our first steps reflect the observation that we can consider the \textit{curved tube} in $X$ as the tubular neighborhood of the defining curve $P$. Fortunately, all relevant theory and technical details can be found in \cite{gray2003tubes}. For the convenience of the reader, we shall summarize the essential tools in the following paragraphs.

Assume that $P$ is a smooth embedded closed curve in the space $X$. Let \(T_pP\) be the tangent space to \(P\) at \(p\in P\) and let us write $\nu_pP:=\left(T_pP\right)^{\perp}=\{v\in T_pX\hspace{0.1cm}\vert\hspace{0.1cm}\langle v,w\rangle=0,\hspace{0.1cm}\forall w\in T_pP\}\subset T_pX$ for the orthogonal complement of the tangent space \(T_pP\). Then the normal bundle $\nu P$ is
\begin{equation*}
    {\nu P}:=\bigsqcup_{p\in P}\nu_pP=  \{(p,v)\hspace{0.1cm}\vert\hspace{0.1cm} p\in P, v\in \nu_pP\}.
\end{equation*}

The following exponential map technique is well-known in the field of differential topology. Our assumption that $T(P,a)$ does not intersect with itself is easily characterized in terms of the exponential map $exp^{\perp}$ of the normal bundle $\nu P$, which is defined as $exp^{\perp} (p,v)=exp_p (v)$, for $(p,v)\in \nu P$, where $exp_p (v)$ is the restriction to the normal bundle. The following proposition ensures the existence of the exponential map in the neighborhood of the zero section of the vector bundle $\nu P$. Moreover, we can then identify $P$ with the zero section of $\nu P$ and consider it as a submanifold of $\nu P$.

\newtheorem{prop}{Proposition}
\begin{prop} \cite[p.~32]{gray2003tubes} \label{prop.1} \textit{Let $P$ be a smooth embedded closed curve in $X$. Then, the exponential map $exp^{\perp}: \nu P \rightarrow X$ restricts to a smooth diffeomorphism of a neighborhood of $P \in \nu P$ onto a neighborhood of $P \in X$}.
\end{prop}

Proposition \ref{prop.1} claims that the function $\rho=\left\| v\right\|$ is given by $\rho=d_p \circ exp^{\perp}$, where $d_p\left(x\right):=\operatorname{dist}\left(x,P\right)$ is the distance function from $x$ to $P$.   

Let us denote $O_P$ the largest neighborhood of the zero section of $\nu P$ such that the exponential map $exp^{\perp}: O_P \rightarrow exp^{\perp} (O_P)$ is a diffeomorphism for that neighborhood.
 
We may reformulate the definition of the \textit{curved tube} as follows 
\begin{definition} \cite[p.33]{gray2003tubes}\label{def:2} 
\textit{Let $P$ be a smooth embedded closed curve in $X$, then the curved tube is defined by} 
\begin{equation*}
T(P,a) := \bigcup_{\forall p\in P} \{exp_p(v)\hspace{0.1cm} \vert\hspace{0.1cm} v\in\nu_pP, \lVert v \rVert \leq a \}.
\end{equation*}
\end{definition} 

We shall also introduce the following definition of the so-called tubular hypersurface.

\begin{definition}\label{def:3}
\textbf {(Tubular hypersurface)} \cite[p.33]{gray2003tubes}  \textit{Let $X$ be a space form with constant sectional curvature $K^X$. Let $P$ be a smooth, closed, embedded curve. Next, let $0\leq r \leq a$. Then a tubular hypersurface about $T(P,a)$ is given by} 
\begin{equation*}
{\delta P_r:=\{x\in T(P,a)\hspace{0.1cm}\vert\hspace{0.1cm} dist(x,P)=r\}}.
\end{equation*}
\end{definition}
It is straightforward from this definition that for $0< r\leq a$ the tubular hypersurfaces form a natural regular foliation of $T(P,a)\setminus P$.

Having introduced the two principal objects used in theorem \ref{thm:1}, we now seek suitable parametrizations for them.  

The choice of the so-called Fermi coordinate system \cite{gray2003tubes}, will be useful for computing the upper bound of $T(P,a)$. Additionally, it is worth mentioning that precisely due to the natural foliation of $T(P,a)\setminus P$ by the tubular hypersurfaces, we can choose the right coordinate system that corresponds exactly to the Fermi coordinates, generalizing the normal coordinates. 

In order to define a system of Fermi coordinates for an arbitrary submanifold $P$, it is necessary to have an arbitrary system of coordinates, say $\left(y_1,...,y_q\right)$ defined in a neighborhood $U\subset P$ of $p\in P$ with orthonormal sections $A_{q+1},...,A_n$ of the restriction of the normal bundle $\nu P$ to $U$. The following definition is generally given for any topologically embedded submanifold $P$ of a Riemannian manifold $M$ \cite{gray2003tubes}. In our case, we will further work with its restriction to $P$ being a smooth embedded closed curve and $M$ being a space form $X$ with constant sectional curvature. For our purposes, we shall restrict our considerations to the case of $q=1$.

\begin{definition} \cite[p.17]{gray2003tubes} 
\textbf{(Fermi coordinates)}\label{def:4}
\textit{Given the orthonormal sections $A_{2},...,A_n$ of the normal bundle $\nu P$, the Fermi coordinates $\left(x_1,...,x_n\right)$ of $T(P,a)$ at $p\in P$ with respect to a given coordinate system $y_1$ on $U\subset P$, parametrizing $P$ by length, are defined as}
\begin{equation} \label{eq:2}
{x_1\left(\operatorname{exp^{\perp}}\left(p,\sum_{i=2}^{n}t_iA_i\left(p\right)\right)\right)=y_1\left(p\right)}
\end{equation}
\begin{equation}\label{eq:3}
{x_{\beta}\left(\operatorname{exp^{\perp}}\left(p,\sum_{i=2}^{n}t_iA_i\left(p\right)\right)\right)=t_{\beta}},
\end{equation}
\end{definition}
\textit{for $p\in U$ and $\beta\in \{2,...,n\}$. Here, the numbers $t_{2},...,t_n$ are small enough that the sum $\sum_{i=2}^n t_iA_i\left(p\right)$ is in $O_P$.}

\vspace{0.1cm}
Recalling the fact that $\operatorname{exp}^{\perp}$ is a diffeomorphism on $O_P$, we can say that (\ref{eq:2}) and (\ref{eq:3}) determine a coordinate system in the neighborhood of $p$.

The next proposition is an obvious consequence of the previous definition. 
\begin{prop} \cite[p.18]{gray2003tubes}\label{prop:2} 
\textit{Let the curve $P$ be parameterized by length. Then, assuming that $\left(x_1,...,x_n\right)$ is a system of Fermi coordinates at $p\in P$, the restrictions of the coordinate vector fields $\frac{\partial}{\partial x_1},...,\frac{\partial}{\partial x_n}$ to $P$ are orthonormal.}
\end{prop}

Fermi coordinates serve as a measuring tool for the geometry of a Riemannian manifold in a close neighborhood of the submanifold (closed curve) $P$. The choice of the coordinate system on $P$ is of no importance. If required, it can be chosen as the system of normal coordinates. Therefore, we are only interested in how $x_1$ and $x_{\beta}$ vary along geodesics that are normal to $P$.

Let \(\nabla\) denote the covariant derivative and assume \(N\) to be the outward radial unit normal to the tubular hypersurface \(\delta P_r\). Then the curvature operator associated with \(N\) is the ambient \((1,1)\)-tensor field
\begin{equation}
R_N:T\left(\operatorname{exp}^{\perp}\left(O_P\right)\setminus P\right)\longrightarrow T\left(\operatorname{exp}^{\perp}\left(O_P\right)\setminus P\right)
\end{equation}
defined by \(\left(R_N\right)_x\left(V\right)=R_x\left(V,N_x\right)N_x\), for all \(x\in X\) and  \(V\in\Xi\left(\operatorname{exp}^{\perp}\left
(O_P\right)\setminus P\right)\), with $\Xi\left(\operatorname{exp}^{\perp}\left(O_P\right)\setminus P\right)$ denoting the set of all vector fields on $\operatorname{exp}^{\perp}\left(O_P\right)\setminus P$.

Now, we introduce the notion of the shape operator.
\begin{definition}\cite[p.33]{gray2003tubes}\label{def:5}
\textit{The shape operator}
\begin{equation}\label{eq:4}
\mathscr{S}:\Xi\left(\operatorname{exp}^{\perp}\left(O_P\right)\setminus P\right)\longrightarrow\Xi\left(\operatorname{exp}^{\perp}\left(O_P\right)\setminus P\right),
\end{equation}
\textit{is defined as}
\begin{equation}\label{eq:5}
\mathscr{S}\left(V\right)=-\nabla_VN,
\end{equation}
\textit{for $V\in\Xi\left(\operatorname{exp}^{\perp}\left(O_P\right)\setminus P\right)$, where $N$ is the outward unit normal vector field.}
\end{definition}

The covariant derivative of the shape operator satisfies an important equation that is essential for computing the second fundamental form of the tubular hypersurface. Consequently, this plays its role within the formula for the volume of $T\left(P,a\right)$.

We state the following lemma.
\newtheorem{lemma}{Lemma}
\begin{lemma}\cite[p.34]{gray2003tubes}\label{lem:1}
\textit{On $\operatorname{exp}^{\perp}\left(O_P\right)-P$, the shape operator satisfies the following equation}
\begin{equation}\label{eq:6}
\nabla_N\left(\mathscr{S}\right)=\mathscr{S}^2+R_N,
\end{equation}
\end{lemma}

Specifically, this tells us that the Riccati equation \(\nabla_N \left(\mathscr{S}\right)=\mathscr{S}^2+R_N\) means pointwise \(\left(\nabla_N\left(\mathscr{S}\right)\right)\left(V\right)=\mathscr{S}^2\left(V\right)+R\left(V,N\right)N\).

\newpage
Now, for each \(0<r\leq a\), let
\begin{equation}
i_r:\delta P_r\longrightarrow X
\end{equation}
denote the inclusion of the tubular hypersurface into the ambient real space form. Recall that \(N\) is the outward radial unit vector field on \(\operatorname{exp}^{\perp}\left(O_P\right)\setminus P\), namely \(N=\nabla d_p\), where \(\tau\left(x\right):=\operatorname{dist}\left(x,P\right)\). Since \(\delta P_r\) is a level set of \(d_p\left(x\right)\), its tangent space is given by
\begin{equation}
T_x\delta P_r=\{v\in T_xX\hspace{0.1cm}\vert\hspace{0.1cm}\langle v,N_x\rangle=0\},
\end{equation}
where \(x\in\delta P_r\) and \(v\in T_x\delta P_r\). Thus \(N_x\) can be considered the unit normal to \(\delta P_r\) at \(x\).

Next, we want to check that \(\mathscr{S}\) could be restricted to \(\delta P_r\). Indeed, when \(V\) is tangent to \(\delta P_r\), we have 
\begin{equation}
\langle \mathscr{S}\left(V\right),N\rangle=-\langle\nabla_V N,N\rangle=-\frac{1}{2}V\langle N,N\rangle=0,
\end{equation}
because \(\vert N\vert=1\). Therefore, \(\mathscr{S}\left(V\right)\) is again tangent to \(\delta P_r\). Hence, the restriction of \(\mathscr{S}\) to \(T\delta P_r\) is well defined. 

Now, we define the map 
\begin{equation}
\mathscr{S}\left(r\right):T\delta P_r\longrightarrow T\delta P_r
\end{equation}
by
\newcommand{\restr}[2]{\ensuremath{\left.#1\right|_{#2}}}
\begin{equation}
\mathscr{S}\left(r\right)_x\left(v\right):=\mathscr{S}_x\left(v\right)=-\nabla_VN,
\end{equation}
where \(x\in\delta P_r\), \(v\in T_x\delta P_r\). Equivalently, we obtain the restriction \(\mathscr{S}\left(r\right)=\restr{\mathscr{S}}{T\delta P_r}\). This is precisely the so-called Weingarten map, or the shape operator of the tubular hypersurface \(\delta P_r\), with respect to the outward unit normal \(N\), as defined in \cite[p.34]{gray2003tubes}. With the sign convention used here, the corresponding second fundamental form is
\begin{equation}
\operatorname{II}_r\left(v,w\right)=\langle \mathscr{S}\left(r\right)v,w\rangle=-\langle\nabla_vN,w\rangle,\quad v,w\in T_x\delta P_r.
\end{equation}

Similarly, according to the previous considerations about \(R_N\), we can restrict it to an endomorphism of the tangent space of the tubular hypersurface \(T\delta P_r\). We define the map
\begin{equation}
R\left(r\right):T_x\delta P_r\longrightarrow T_x\delta P_r,
\end{equation}
by \(R\left(r\right)_x\left(v\right):=\left(R_N\right)_x\left(v\right)=R_x\left(v,N_x\right)N_x\), where \(x\in\delta P_r\), \(v\in T_x\delta P_r\). This restriction is well defined because for every \(v\in T_x\delta P_r\), we have \(\langle R_N\left(v\right),N\rangle=\langle R\left(v,N\right)N,N\rangle=R\left(v,N,N,N\right)=0\), where the last equality follows from the skew-symmetry of the curvature tensor in its last two entries. Thus, \(R_N\left(v\right)\) is orthogonal to \(N\), hence, it is tangent to the hypersurface \(\delta P_r\) whenever \(v\perp N\). 

In this sense, \(\mathscr{S}\left(r\right)\) and \(R\left(r\right)\) are the restrictions of the tensor fields \(\mathscr{S}\) and \(R_N\) to the tangent bundle of the level hypersurface \(\delta P_r\). 

According to these conventions, the Riccati equation from lemma \ref{lem:1} restricts along \(\delta P_r\) to the following hypersurface equation.

\begin{lemma}\cite[p.34]{gray2003tubes}\label{lem:2}
\textit{The shape operator of $\delta P_r$ satisfies}
\begin{equation}\label{eq:7}
\mathscr{S}^{\prime}\left(r\right)=\mathscr{S}^2\left(r\right)+R\left(r\right).
\end{equation}
\end{lemma}

Here \(\mathscr{S}^{\prime}\left(r\right)\) means the covariant derivative of the family \(\mathscr{S}\left(r\right)\) in the radial direction \(N\). More precisely, if \(x=\operatorname{exp}_p\left(ru\right)\in\delta P_r\), and if \(v\left(r\right)\) is a vector field along the radial geodesic \(r\mapsto\operatorname{exp}_p\left(ru\right)\) with \(v\left(r\right)\in T_{\operatorname{exp}_p\left(ru\right)}\delta P_r\), then the term \(\mathscr{S}^{\prime}\left(r\right)v\left(r\right)\) is represented by \(\left(\nabla_N\mathscr{S}\right)\left(v\left(r\right)\right)\). Equivalently, we have \(\mathscr{S}^{\prime}\left(r\right)_x\left(v\right):=\left(\nabla_N\mathscr{S}\right)_x\left(v\right)\), with \(x\in\delta P_r\) and \(v\in T_x\delta P_r\). Thus, the equation (\ref{eq:7}) can be viewed as an equality of endomorphisms of \(T\delta P_r\). 

Henceforth, the longitudinal component of \(\mathscr{S}\left(r\right)\) admits a finite limit as \(r\longrightarrow 0\), determined by the second fundamental form of \(P\). However, this limit must be taken along a specific normal geodesic to make sense.

For a real space form of constant sectional curvature \(K^X\), the curvature tensor satisfies
\begin{equation}
R\left(X,Y\right)Z=K^X\left(\langle Z,Y\rangle X-\langle Z,X\rangle Y\right).
\end{equation}
Now, take \(V\in T_x\delta P_r\). Since \(N_x\) is normal to \(\delta P_r\), we have \(\langle V,N\rangle =0\). Therefore,
\begin{equation}
R\left(V,N\right)N=K^X\left(\langle N,N\rangle V-\langle N,V\rangle N\right)=K^X\left(V-0\right)=K^XV.
\end{equation}
Hence, for every \(V\in T_x\delta P_r\), we obtain
\begin{equation}
R\left(r\right)_x\left(V\right)=K^XV.
\end{equation}
Therefore,
\begin{equation}
R\left(r\right)_x=K^X\operatorname{Id}_{T_x\delta P_r}
\end{equation}
and globally
\begin{equation}
R\left(r\right)=K^X\operatorname{Id}_{T_\delta P_r}.
\end{equation}

Additionally, we state a useful property that each of the eigenvalues of $\mathscr{S}\left(r\right)$ satisfies the same scalar differential equation as (\ref{eq:7}). These eigenvalues are the so-called principal curvatures $\kappa_i\left(r\right)$, with $i\in\{1, 3, ...,n\}$. 

\begin{lemma}\cite[p.51]{gray2003tubes}\label{lem:3}
\textit{Let $\mathscr{S}$ be the shape operator of the tubular hypersurface $\delta P_r$. Then the principal curvatures $\kappa_i\left
(r\right)$ satisfy the following differential equation}
\begin{equation}\label{eq:8}
\kappa_i^{\prime}\left(r\right)=\kappa^2_i\left(r\right)+K^X,
\end{equation}
\textit{where $i\in\{1,3,...,n\}$ and $K^X$ represents the constant sectional curvature of the space form $X$.}
\end{lemma}

\section{Proof of theorem \ref{thm:1}}

We are ready to proceed further and derive the upper bound of $h(T(P,a))$ for $K^X>0$ and $K^X<0$. Dealing with the tubes in real space forms allows us to effectively exploit the general tools provided in \cite{gray2003tubes}.

\subsection{The upper bound}\label{sec.3.1}
\vspace{0.2cm}

Obviously, straight from the definition of $h(T(P,a))$, we know the estimate
\begin{equation*}
h\left(T\left(P,a\right)\right)\leq \frac{\left| \partial T(P,a)\right|}{\left| T(P,a)\right|}.
\end{equation*}
This elementary estimate can be used in all three cases of space forms to provide the upper bound. The upper bound for the case $K^X=0$ was computed in \cite{krejcirik2019cheeger}.

Let us briefly recall that the Riemannian volume form on $X$ is an $n$-form $\omega$. To every orthonormal frame, this volume form assigns the values $\pm 1$. In the following definition, we introduce the so-called volume function $\psi_u\left(r\right)$ that measures the distortion of the volume of $T\left(P,a\right)$ in the direction of the vector $u\in \left(T_pP\right)^{\perp}$.

Moreover, the subsequent definition describes the volume function $\psi_u\left(r\right)$ in terms of Fermi coordinates. 

\begin{definition}\cite[p.38]{gray2003tubes}\label{def:6}
\textit{Let $p\in P$ and $u\in \left(T_pP\right)^{\perp}$ be a unit vector. Then the volume function $\psi_u\left(r\right)$ is defined as}
\begin{equation}\label{eq:9}
\psi_u\left(r\right)=\omega\left(\frac{\partial}{\partial x_1}\wedge\hspace{0.1cm}...\hspace{0.1cm}\wedge\frac{\partial}{\partial x_n}\right)\left(\operatorname{exp}^{\perp}\left(p,ru\right)\right),
\end{equation}
\textit{where $0<r\leq a$ and $\left(p,ru\right)\in O_P$.}
\end{definition}

Definition \ref{def:6} tells us how the exponential map $\operatorname{exp}_{\eta}$ infinitesimally distorts the volume of $T(P,a)$. This distortion is measured by the infinitesimal change of volume function $\psi_u\left(r\right)$ with respect to the system of Fermi coordinates.

The following lemma provides the formula for the volume of $T\left(P,a\right)$.

\begin{lemma}\cite[p. 42]{gray2003tubes}\label{lem:4}
\textit{Since $\operatorname{exp}^{\perp}:\nu P\rightarrow X$ of a neighborhood of $P\in\nu P$ onto a neighborhood of $P\in X$, the volume function $\psi_u\left(r\right)$ determines the volume of $T(P,a)$ as}
\begin{equation}\label{eq:11}
    {\vert T(P,a)\vert=\int_0^a\int_P\int_{\mathbb{S}^{n-2}\left(1\right)}r^{n-2}\psi_u\left(r\right)\operatorname{du}\operatorname{dP}\operatorname{dr}},
\end{equation}
\textit{where $\psi_u\left(r\right)$ is given by the defining relation (\ref{eq:9}), $\operatorname{du}$ represents the volume element of the unit sphere $\mathbb{S}^{n-2}\left(1\right)$ and $\operatorname{dP}$ is the length element of $P$.}
\end{lemma}

In the following proposition, we will state the fundamental differential equation that provides an implicit definition of the volume function $\psi_u\left(r\right)$, with respect to the shape operator $\mathscr{S}\left(r\right)$. 

\begin{prop}\cite[p.39]{gray2003tubes}\label{prop:3}
\textit{Let $\alpha$ be a unit speed geodesic in $\operatorname{exp}_{\eta}\left(O_P\right)$ that is normal to $P$, with $\alpha\left(0\right)=p\in P$ and $\alpha^{\prime}\left(0\right)=u\in \left(T_pP\right)^{\perp}$. Then, along the geodesic $\alpha$, the volume function $\psi_u\left(r\right)$ satisfies the following differential equation}
\begin{equation}\label{eq:10}
\frac{\psi^{\prime}_u\left(r\right)}{\psi_u\left(r\right)}=-\frac{n-2}{r}-\operatorname{Tr}\left(\mathscr{S}\left(r\right)\right),
\end{equation}
\textit{for $0<r\leq a$ and $\left(p,ru\right)\in O_P$, with the initial condition $\psi_u\left(0\right)=1$.}
\end{prop}
   
It is relatively straightforward to prove that the formula for the $\left(n-1\right)$-dimensional volume of $\partial T(P,a)$ satisfies $\frac{\partial}{\partial a}\vert T(P,a)\vert=\vert\partial T(P,a)\vert$ \cite[p.42]{gray2003tubes}.

In order to exploit the proposition \ref{prop:3}, we shall first compute the principal curvatures $\kappa_i\left(r\right)$ from which we subsequently obtain the trace of the shape operator $\operatorname{Tr}\left(\mathscr{S}\left(r\right)\right)$. This will allow us to compute $\psi_u\left(r\right)$ as a solution of (\ref{eq:10}).

Therefore, we must first prove the following lemma.

\begin{lemma}\label{lem:5}
\textit{The principal curvatures of the tubular hypersurface $\delta P_r$ about $P$ satisfy}
\begin{equation}\label{eq:12}
\kappa_1\left(r\right)=
\begin{cases}
\begin{aligned}
\frac{\sqrt{K^X}\operatorname{sin}\left(\sqrt{K^X}r\right)+\kappa_1\left(0\right)\operatorname{cos}\left(\sqrt{K^X}r\right)}{\operatorname{cos}\left(\sqrt{K^X}r\right)-\frac{\kappa_1\left(0\right)}{\sqrt{K^X}}\operatorname{sin}\left(\sqrt{K^X}r\right)},\hspace{0.2cm} K^X>0 \\[1ex]
\frac{\kappa_1\left(0\right)}{1-\kappa_1\left(0\right)r}, \hspace{0.2cm}K^X=0 \\[1ex]
\frac{\sqrt{\vert K^X\vert}\operatorname{sinh}\left(\sqrt{\vert K^X\vert}r\right)+\kappa_1\left(0\right)\operatorname{cosh}\left(\sqrt{\vert K^X\vert}r\right)}{\operatorname{cosh}\left(\sqrt{\vert K^X\vert}r\right)-\frac{\kappa_1\left(0\right)}{\sqrt{\vert K^X\vert}}\operatorname{sinh}\left(\sqrt{\vert K^X\vert}r\right)},\hspace{0.2cm}K^X<0
\end{aligned}
\end{cases}
\end{equation}
\textit{where $\kappa_1\left(0\right)$ is finite,}
\begin{equation}\label{eq:13}
\kappa_i\left(r\right)=
\begin{cases}
\begin{aligned}
-\sqrt{K^X}\operatorname{cotg}\left(\sqrt{K^X}r\right),\hspace{0.2cm}K^X>0 \\[1ex]
-\frac{1}{r}, \hspace{0.2cm}K^X=0 \\[1ex]
-\sqrt{\vert K^X\vert}\operatorname{cotgh}\left(\sqrt{\vert K^X\vert
}r\right),\hspace{0.2cm}K^X<0
\end{aligned}
\end{cases},
\end{equation}
\textit{for $i\in\{3,...,n\}$, with $\kappa_i\left(0\right)=-\infty$.}
\end{lemma}

\hspace{-0.7cm}
\vspace{0.2cm}
\textbf{Proof of lemma \ref{lem:5}.}

Notice that all the principal curvatures are differentiable functions. Then if we consider the first case when $K^X>0$, we have
\begin{equation}\label{eq:14}
\kappa^{\prime}_i\left(r\right)=\kappa^2_i\left(r\right)+K^X.
\end{equation}
This equation can be solved via substitution $\kappa_i\left(r\right)=-\frac{u^{\prime}\left(r\right)}{u\left(r\right)}$, where \(u\left(r\right)\neq 0\), providing us with the following transformation

\begin{equation}\label{eq:15}
-\frac{u^{\prime\prime}\left(r\right)u\left(r\right)-\left(u^{\prime}\left(r\right)\right)^2}{u^2\left(r\right)}=\left(\frac{u^{\prime}\left
(r\right)}{u\left(r\right)}\right)^2+K^X,
\end{equation}
\begin{equation}\label{eq:16}
-u^{\prime\prime}\left(r\right)u\left(r\right)-K^Xu^2\left(r\right)=0.
\end{equation}
We have arrived at a second-order differential equation with constant coefficients
\begin{equation}\label{eq:17}
u^{\prime\prime}\left(r\right)+K^Xu\left(r\right)=0.
\end{equation}

Writing $K^X=\alpha^2$, where $\alpha>0$, the solution of (\ref{eq:17}) is given by

\begin{equation}\label{eq:18}
u\left(r\right)=A\operatorname{cos}\left(\alpha r\right)+B\operatorname{sin}\left(\alpha r\right),
\end{equation}
where \(A,B\) are the integration constants.

By plugging the above result back into the substitution, we attain
\begin{equation}\label{eq:19}
\kappa_i\left(r\right)=-\frac{-A\alpha\operatorname{sin}\left(\alpha r\right)+B\alpha\operatorname{cos}\left(\alpha r\right)}{A\operatorname{cos}\left(\alpha r\right)+B\operatorname{sin}\left(\alpha r\right)}.
\end{equation}

Since $K^X=\alpha^2$, we can rewrite the general solution as
\begin{equation}\label{eq:20}
\kappa_i\left(r\right)=\frac{\sqrt{K^X}\left(A\operatorname{sin}\left(\sqrt{K^X} r\right)-B\operatorname{cos}\left(\sqrt{K^X} r\right)\right)}{A\operatorname{cos}\left(\sqrt{K^X} r\right)+B\operatorname{sin}\left(\sqrt{K^X} r\right)}.    
\end{equation}

The initial condition is determined by 
\begin{equation}\label{eq:21}
\kappa_i\left(0\right)=\frac{-B\sqrt{K^X}}{A}.
\end{equation}
Plugging this back into (\ref{eq:20}) gives us the complete solution for the principal curvatures of $\delta P_r$
\begin{equation}\label{eq:22}
\kappa_i\left(r\right)=\frac{\sqrt{K^X}\operatorname{sin}\left(\sqrt{K^X}r\right)+\kappa_i\left(0\right)\operatorname{cos}\left(\sqrt{K^X}r\right)}{\operatorname{cos}\left(\sqrt{K^X}r\right)-\frac{\kappa_i\left(0\right)}{\sqrt{K^X}}\operatorname{sin}\left(\sqrt{K^X}r\right)}.
\end{equation}

The formula (\ref{eq:22}) describes the solution of the Riccati equation with a finite initial value \(\kappa_i\left(0\right)\). However, for the tubular hypersurface \(\delta P_r\) around \(P\), its tangent space \(T_x\delta P_r\) is composed of one longitudinal direction and \(n-2\) angular directions. The longitudinal direction comes from moving the base point \(p\in P\) and this direction survives as \(r\longrightarrow 0\), because it converges to \(T_p P\). The angular directions are different. They come from changing the normal direction \(u\in \mathbb{S}\left(\left(T_pP\right)^{\perp}\right)\cong \mathbb{S}^{n-2}\). But when \(r=0\), all points \(\operatorname{exp}_p\left(ru\right)\) collapse to the same point \(p\), independently of \(u\). 

Therefore, the angular sphere collapses to a point at \(r=0\). Hence, the angular directions do not define genuine tangent directions of \(P\). They are tangent to the small normal spheres around \(p\), and those spheres shrink to zero size at \(r\longrightarrow 0\).

Consequently, the corresponding principal curvatures cannot have finite initial values. Exactly as for an ordinary Euclidean sphere of radius \(r\), the angular principal curvatures behave like \(-1/r\). This means that the appropriate initial condition is \(\kappa_i\left(0\right)=-\infty\) for all \(i\in\{3,...,n\}\). 

Therefore, the angular principal curvatures are obtained by taking the singular limit \(\kappa_i\left(0\right)\longrightarrow -\infty\) in (\ref{eq:22}). We obtain

\begin{equation*}\label{eq.36}
\lim_{\kappa_i\left(0\right)\longrightarrow-\infty}\frac{\sqrt{K^X}\operatorname{sin}\left(\sqrt{K^X}r\right)+\kappa_i\left(0\right)\operatorname{cos}\left(\sqrt{K^X}r\right)}{\operatorname{cos}\left(\sqrt{K^X}r\right)-\frac{\kappa_i\left(0\right)}{\sqrt{K^X}}\operatorname{sin}\left(\sqrt{K^X}r\right)}=\lim_{\kappa_i\left(0\right)\rightarrow \infty}\frac{\frac{\sqrt{K^X}\operatorname{sin}\left(\sqrt{K^X}r\right)}{\kappa_i\left(0\right)}+\operatorname{cos}\left(\sqrt{K^X}r\right)}{\frac{\operatorname{cos}\left(\sqrt{K^X}r\right)}{\kappa_i\left(0\right)}-\frac{1}{\sqrt{K^X}}\operatorname{sin}\left(\sqrt{K^X}r\right)}=
\end{equation*}
\begin{equation}\label{eq:23}
\hspace{7.5cm}
=-\sqrt{K^X}\operatorname{cotg}\left(\sqrt{K^X}r\right).
\end{equation}

In the case of \(i=1\) and a finite initial condition $\kappa_1\left(0\right)$, we can consider \(\kappa_1\left(r\right)\) as the first principal curvature in the tangential direction to the underlying curve
\begin{equation}\label{eq:24}
\kappa_1\left(r\right)=\frac{\sqrt{K^X}\operatorname{sin}\left(\sqrt{K^X}r\right)+\kappa_1\left(0\right)\operatorname{cos}\left(\sqrt{K^X}r\right)}{\operatorname{cos}\left(\sqrt{K^X}r\right)-\frac{\kappa_1\left(0\right)}{\sqrt{K^X}}\operatorname{sin}\left(\sqrt{K^X}r\right)},
\end{equation}
where $\kappa_1\left(0\right)$ is the normal curvature of $P$ in the chosen normal direction \(u\). 

Straightforwardly, the Euclidean case can be derived from the following limit
\begin{equation}\label{eq:26}
\lim_{K^X\rightarrow 0}\frac{\sqrt{K^X}\operatorname{sin}\left(\sqrt{K^X}r\right)+\kappa_1\left(0\right)\operatorname{cos}\left(\sqrt{K^X}r\right)}{\operatorname{cos}\left(\sqrt{K^X}r\right)-\frac{\kappa_1\left(0\right)}{\sqrt{K^X}}\operatorname{sin}\left(\sqrt{K^X}r\right)}=\frac{\kappa_1\left(0\right)}{1-\kappa_1\left(0\right)r}.
\end{equation}

This corresponds to the first principal curvature of $\delta P_r$ along the tangential direction to $P$, where its curvature $\kappa_1\left(0\right)$ is considered finite. 

Analogously, the principal curvatures in the normal directions to $P$ can be obtained by
\begin{equation}\label{eq:27}
\lim_{\kappa_1\left(0\right)\rightarrow -\infty}\frac{\kappa_1\left(0\right)}{1-\kappa_1\left(0\right)r}=\lim_{\kappa_1\left(0\right)\rightarrow -\infty}\frac{1}{\frac{1}{\kappa_1\left(0\right)}-r}=-\frac{1}{r}.
\end{equation}

Just for clarity, the identical result can be obtained by taking the limit of (\ref{eq.36}), with $K^X\rightarrow 0$.  

Finally, in the hyperbolic case with $K^X<0$, the proof remains nearly identical. The only requirement is to replace the functions $\operatorname{sin}, \operatorname{cos}$ and $\operatorname{cotg}$ by their hyperbolic variants and substitute $K^X$ by $\vert K^X\vert$. 

This completes the proof of lemma \ref{lem:5}.

\vspace{0.2cm}
The result of lemma \ref{lem:5} enables us to compute the trace of $\mathscr{S}\left(r\right)$ straightforwardly from the relations (\ref{eq:12}) and (\ref{eq:13}). The trace is given by the summation of $\kappa_i\left(r\right)$ over all $i\in\{1,...,n\}$, where $i\neq 2$. Therefore
\begin{equation}\label{eq:28}
\hspace{-0.8cm}
\operatorname{Tr}\left(\mathscr{S}\left(r\right)\right)=
\begin{cases}
\begin{aligned}
\frac{\sqrt{K^X}\operatorname{sin}\left(\sqrt{K^X}r\right)+\kappa_1\left(0\right)\operatorname{cos}\left(\sqrt{K^X}r\right)}{\operatorname{cos}\left(\sqrt{K^X}r\right)-\frac{\kappa_1\left(0\right)}{\sqrt{K^X}}\operatorname{sin}\left(\sqrt{K^X}r\right)}-\left(n-2\right)\sqrt{K^X}\operatorname{cotg}\left(\sqrt{K^X}r\right), \hspace{0.2cm} K^X>0 \\[1ex]
\frac{\kappa_1\left(0\right)}{1-\kappa_1\left
(0\right)r}-\frac{n-2}{r}, \hspace{0.2cm} K^X=0 \\[1ex]
\frac{\sqrt{\vert K^X\vert}\operatorname{sinh}\left(\sqrt{\vert K^X\vert}r\right)+\kappa_1\left(0\right)\operatorname{cosh}\left(\sqrt{\vert K^X\vert}r\right)}{\operatorname{cosh}\left(\sqrt{\vert K^X\vert}r\right)-\frac{\kappa_1\left(0\right)}{\sqrt{\vert K^X\vert}}\operatorname{sinh}\left(\sqrt{\vert K^X\vert}r\right)}-\left(n-2\right)\sqrt{\vert K^X\vert}\operatorname{cotgh}\left(\sqrt{\vert K^X\vert}r\right), \hspace{0.2cm} K^X<0.
\end{aligned}
\end{cases}
\end{equation}

This leads us to the following lemma, where we solve the differential equation for the volume function $\psi_u\left(r\right)$.
\begin{lemma}\label{lem:6}
\textit{The volume function $\psi_u\left(r\right)$ satisfying the differential equation (\ref{eq:10}) with the initial condition $\psi_u\left(0\right)=1$, is given by}
\begin{equation}\label{eq:29}
\psi_u\left(r\right)=
\begin{cases}
\begin{aligned}
\left(\frac{\operatorname{sin}\left(\sqrt{K^X}r\right)}{\sqrt{K^X}r}\right)^{n-2}\left(\operatorname{cos}\left(\sqrt{K^X}r\right)-\frac{\kappa_1\left(0\right)}{\sqrt{K^X}}\operatorname{sin}\left(\sqrt{K^X}r\right)\right), \hspace{0.2cm} K^X>0 \\[1ex]
1-\kappa_1\left(0\right)r, \hspace{0.2cm} K^X=0 \\[1ex]
\left(\frac{\operatorname{sinh}\left(\sqrt{\vert K^X\vert}r\right)}{\sqrt{\vert K^X\vert}r}\right)^{n-2}\left(\operatorname{cosh}\left(\sqrt{\vert K^X\vert}r\right)-\frac{\kappa_1\left(0\right)}{\sqrt{\vert K^X\vert}}\operatorname{sinh}\left(\sqrt{\vert K^X\vert}r\right)\right), \hspace{0.2cm} K^X<0
\end{aligned}
\end{cases}
\end{equation}
\end{lemma}

\vspace{0.2cm}

\hspace{-0.6cm}\textbf{Proof of lemma \ref{lem:6}.}

Let us first focus on the spherical case when $K^X>0$. We can plug the formula (\ref{eq:28}) back into the differential equation (\ref{eq:10}), this gives us
\begin{equation}\label{eq:30}
\frac{\psi^{\prime}_u\left(r\right)}{\psi_u\left(r\right)}=-\frac{n-2}{r}-\frac{\sqrt{K^X}\operatorname{sin}\left(\sqrt{K^X}r\right)+\kappa_1\left(0\right)\operatorname{cos}\left(\sqrt{K^X}r\right)}{\operatorname{cos}\left(\sqrt{K^X}r\right)-\frac{\kappa_1\left(0\right)}{\sqrt{K^X}}\operatorname{sin}\left(\sqrt{K^X}r\right)}+\left(n-2\right)\sqrt{K^X}\operatorname{cotg}\left(\sqrt{K^X}r\right).
\end{equation}

It is evident that the structure of the differential equation corresponds to the first-order ODE with the right-hand side as a function of $r$
\begin{equation}\label{eq:31}
\frac{\psi^{\prime}_u\left(r\right)}{\psi_u\left(r\right)}=f\left(r\right),
\end{equation}
where $f\left(r\right)=-\frac{n-2}{r}-\frac{\sqrt{K^X}\operatorname{sin}\left(\sqrt{K^X}r\right)+\kappa_1\left(0\right)\operatorname{cos}\left(\sqrt{K^X}r\right)}{\operatorname{cos}\left(\sqrt{K^X}r\right)-\frac{\kappa_1\left(0\right)}{\sqrt{K^X}}\operatorname{sin}\left(\sqrt{K^X}r\right)}+\left(n-2\right)\operatorname{cotg}\left(\sqrt{K^X}r\right)$. Thus, the solution of (\ref{eq:31}) is given by
\begin{equation}\label{eq:32}
\psi_u\left(r\right)=C\operatorname{exp}\left({\int f\left(r\right)\operatorname{dr}}\right),
\end{equation}
where $C\in\mathbb{R}$ is the integration constant.

Therefore, we obtain
\begin{equation}\label{eq:33}
\hspace{2.5cm}
\psi_u\left(r\right)=C\left(r^{2-n}\operatorname{sin}^{n-2}\left(\sqrt{K^X}r\right)\left(\operatorname{cos}\left(\sqrt{K^X}r\right)-\frac{\kappa_1\left(0\right)}{\sqrt{K^X}}\operatorname{sin}\left(\sqrt{K^X}r\right)\right)\right),
\end{equation}
where the integration constant $C\in\mathbb{R}$ depends on the initial condition.

Now we compute the limit of (\ref{eq:33}) as \(r\longrightarrow 0\). The first term yields
\begin{equation*}
\lim_{r\rightarrow 0}r^{2-n}\operatorname{sin}^{n-2}\left(\sqrt{K^X}r\right)=\lim_{r\rightarrow 0}\left(\frac{\operatorname{sin}\left(\sqrt{K^X}r\right)}{r}\right)^{n-2}=\left(\sqrt{K^X}\right)^{n-2},
\end{equation*}
and the mixed bracket in (\ref{eq:33}) becomes $1$. Therefore
\begin{equation}
\lim_{r\longrightarrow 0}\psi_u\left(r\right)=C\left
(\sqrt{K^X}\right)^{n-2}. 
\end{equation}

The initial condition is \(\psi_u\left(0\right)=1\). Hence \(C\left(\sqrt{K^X}\right)^{n-2}=1\), so \(C=\left(\sqrt{K^X}\right)^{2-n}\). Then the volume function is
\begin{equation}\label{eq:34}
\psi_u\left(r\right)=\left(\frac{\operatorname{sin}\left(\sqrt{K^X}r\right)}{\sqrt{K^X}r}\right)^{n-2}\left(\operatorname{cos}\left(\sqrt{K^X}r\right)-\frac{\kappa_1\left(0\right)}{\sqrt{K^X}}\operatorname{sin}\left(\sqrt{K^X}r\right)\right).
\end{equation}

In the Euclidean case, the differential equation (\ref{eq:10}) becomes
\begin{equation}
\frac{\psi^{\prime}_u\left(r\right)}{\psi_u\left(r\right)}=-\frac{\kappa_1\left(0\right)}{1-\kappa_1\left(0\right)r}.
\end{equation}
The solution is given by
\begin{equation}
\psi_u\left(r\right)=1-\kappa_1\left(0\right)r.
\end{equation}

Analogously, the proof for the hyperbolic case remains identical. One just needs to interchange the trigonometric functions in (\ref{eq:34}) with their hyperbolic counterparts and swap the sectional curvature $K^X$ with $\vert K^X\vert$ inside the square root.

This completes the proof of lemma \ref{lem:6}.

\vspace{0.2cm}
Based on the result of lemma \ref{lem:6}, we are now able to compute the volume of $T\left(P,a\right)$ and $\vert\partial T\left(P,a\right)\vert$. Nonetheless, it is worth noting that the result of lemma \ref{lem:4} actually tells us that we can reconsider the integration $\int_0^a\int_{\mathbb{S}^{n-2}\left(1\right)}\psi_u\left(r\right)r^{n-2}\operatorname{dr}$ as the integration of $\psi\left(r\right)$ over a ball with radius $a>0$. 

More precisely, if \(g\left(r\right)\) is a radial function, i.e. it depends only on the radial distance \(r\) from the origin, then the integration of $g\left(r\right)$ over the $\left(n-1\right)$-dimensional normal ball $B\left(a\right)$ with radius $a>0$ can be reduced to a one-dimensional integral from $0$ to $a$, involving the surface measure $\vert\mathbb{S}^{n-2}\left(1\right)\vert$ and the factor $r^{n-2}$. Therefore,
\begin{equation}\label{eq:37}
\vert T\left(P,a\right)\vert=\int_P\int_{B\left(a\right)}\psi_u\left(r\right)\operatorname{dV_{B\left(a\right)}}\operatorname{dP}.
\end{equation}

At this point, we directly pass to polar coordinates in the normal ball \(B\left(a\right)\), where the volume element is given by
\begin{equation}
\operatorname{dV_{B\left(a\right)}}=\vert\mathbb{S}^{n-2}\left(1\right)\vert r^{n-2}\operatorname{dr}.
\end{equation}
Consequently, we have
\begin{equation}
\vert T\left(P,a\right)\vert=\vert P\vert\vert\mathbb{S}^{n-2}\left(1\right)\vert\int_0^a r^{n-2}\psi_u\left(r\right)\operatorname{dr}.
\end{equation}

First, picking the spherical case $K^X>0$, the volume of $T\left(P,a\right)$ can be computed as 
\begin{equation}\label{eq:38}
\begin{split}
\vert T\left(P,a\right)\vert=\vert P\vert\vert\mathbb{S}^{n-2}\left(1\right)\vert\int_0^ar^{n-2}\left(\frac{\operatorname{sin}\left(\sqrt{K^X}r\right)}{\sqrt{K^X}r}\right)^{n-2}\left(\operatorname{cos}\left(\sqrt{K^X}r\right)-\frac{\kappa_1\left(0\right)}{\sqrt{K^X}}\operatorname{sin}\left(\sqrt{K^X}r\right)\right)\operatorname{dr} \\
=\frac{\vert P\vert\vert\mathbb{S}^{n-2}\left(1\right)\vert}{\left(\sqrt{K^X}\right)^{n-2}}\int_0^a\operatorname{sin}^{n-2}\left(\sqrt{K^X}r\right)\left(\operatorname{cos}\left(\sqrt{K^X}r\right)-\frac{\kappa_1\left(0\right)}{\sqrt{K^X}}\operatorname{sin}\left(\sqrt{K^X}r\right)\right)\operatorname{dr}.
\end{split}
\end{equation}

Since \(P\) is parametrized by arc-length \(\gamma\left(s\right)\in P\), for a fixed \(s\in P\), we can define the curvature vector \(A\left(s\right):=\nabla_T T\left(s\right)\in P_{\gamma\left(s\right)}^{\perp}\), where \(T\) is the unit tangent vector field of the curve \(P\), i.e. \(T=\gamma^{\prime}\left(s\right)\). Then we have \(\kappa_1\left(0\right):=\kappa_u\left(0;s\right)=\langle A\left(s\right),u\rangle\). Therefore, when integrating the principal curvature of \(P\) over \(\mathbb{S}^{n-2}\left(1\right)\), we obtain
\begin{equation}
\int_{\mathbb{S}^{n-2}}\kappa_1\left(0\right)\operatorname{du}=\int_{\mathbb{S}^{n-2}}\kappa_u\left(s\right)\operatorname{du}=\int_{\mathbb{S}^{n-2}}\langle A\left(s\right),u\rangle\operatorname{du}.
\end{equation}
Since \(A\left(s\right)\) is fixed with respect to the angular integration, we get
\begin{equation}
\int_{\mathbb{S}^{n-2}}\langle A\left(s\right),u\rangle\operatorname{du}=\Bigg\langle A\left(s\right),\int_{\mathbb{S}^{n-2}}u\operatorname{du} \Bigg\rangle.
\end{equation}
However, the integral over \(u\) is identically zero. This follows from the antipodal symmetry, i.e. for every \(u\in\mathbb{S}^{n-2}\), the point \(-u\) also belongs to \(\mathbb{S}^{n-2}\) and the measure is invariant under \(u\mapsto -u\). Hence, 
\begin{equation}
\int_{\mathbb{S}^{n-2}}\kappa_1\left(0\right)\operatorname{du}=0,
\end{equation}
so the second integral in (\ref{eq:38}) vanishes. 

This gives us
\begin{equation}
\vert T\left(P,a\right)\vert=\frac{\vert P\vert\vert\mathbb{S}^{n-2}\left(1\right)\vert}{\left(\sqrt{K^X}\right)^{n-2}}\int_0^a\operatorname{sin}^{n-2}\left(\sqrt{K^X}r\right)\operatorname{cos}\left(\sqrt{K^X}r\right)\operatorname{dr}.
\end{equation}

This integral can be solved by utilizing the substitution $v=\operatorname{sin}\left(\sqrt{K^X}r\right)$, so $\operatorname{d}v=\sqrt{K^X}\operatorname{cos}\left(\sqrt{K^X}r\right)\operatorname{dr}$. At the end, we obtain
\begin{equation*}
\int_0^a\operatorname{sin}^{n-2}\left(\sqrt{K^X}r\right)\operatorname{cos}\left(\sqrt{K^X}r\right)=\frac{\operatorname{sin}^{n-1}\left(\sqrt{K^X}a\right)}{\left(n-1\right)\sqrt{K^X}}.
\end{equation*}

Hence 
\begin{equation}\label{eq:39}
\vert T\left(P,a\right)\vert=\frac{\vert P\vert\vert\mathbb{S}^{n-2}\left(1\right)\vert\operatorname{sin}^{n-1}\left(\sqrt{K^X}a\right)}{\left(n-1\right)\left(\sqrt{K^X}\right)^{n-1}}.
\end{equation}

\vspace{0.2cm}
Since the exponential map $\operatorname{exp}^{\perp}$ is an isometry of a neighborhood of $P\in\nu P$ onto a neighborhood of $P\in X$, we have $\vert\partial T\left(P,a\right)\vert=\frac{\partial}{\partial a}\vert T\left(P,a\right)\vert$,  \cite[p.~42]{gray2003tubes}. This gives us
\begin{equation}\label{eq:40}
\vert\partial T\left(P,a\right)\vert=\frac{\partial}{\partial a}\vert T\left(P,a\right)\vert=\frac{\vert P\vert\vert \mathbb{S}^{n-2}\left(1\right)\vert}{\left(\sqrt{K^X}\right)^{n-2}}\operatorname{sin}^{n-2}\left(\sqrt{K^X}a\right)\operatorname{cos}\left(\sqrt{K^X}a\right)
\end{equation}

This enables us to compute the upper bound of $h\left(T\left(P,a\right)\right)$ when $K^X>0$ as
\begin{equation}\label{eq:41}
h\left(T\left(P,a\right)\right)\leq\left(n-1\right)\sqrt{K^X}\operatorname{sin}^{n-2+1-n}\left(\sqrt{K^X}a\right)\operatorname{cos}\left(\sqrt{K^X}a\right)=\left(n-1\right)\sqrt{K^X}\operatorname{cotg}\left(\sqrt{K^X}a\right).
\end{equation}

Considering the opposite case with the sectional curvature satisfying $K^X<0$, the only required change is to replace the usual $\operatorname{cos}$ and $\operatorname{sin}$ functions with their hyperbolic versions and also take $K^X$ in the absolute value. The rest of the procedure remains identical as in the previous case. 

Therefore, we obtain the following upper bound
\begin{equation}\label{eq:42}
  {h\left(T(P,a)\right)\leq\left(n-1\right)\sqrt{\vert K^X\vert}\operatorname{cotgh}\left(\sqrt{\vert K^X\vert}a\right)}. 
\end{equation} 

\subsection{The lower bound}
\vspace{5px}

In this section we compute the lower bound via the divergence-calibration argument, which is based on the divergence theorem. The vector field used below is not the radial field associated with the distance function \(x\mapsto\operatorname{dist}\left(x,P\right)\). Instead, we use a moving-geodesic-sphere field, such that for sufficiently small \(a\), points of the tube are represented as lying on geodesic spheres of fixed radius \(a\) centered along \(P\), and the calibrating vector field is chosen as the outward radial unit field of these spheres. Upon that, we shall demonstrate that \(T\left(P,a\right)\) is self-Cheeger.

Under the assumption that we have a compact smooth Riemannian manifold $M$ and the given submanifold $N$ has a smooth boundary $\partial N$ with compact closure, it is not obligatory to deal with only smooth vector fields. Let us then define the following class of vector fields 

\begin{equation*}
V_{\operatorname{div}}\left(N\right):=\{V\in L^{\infty}\left(N;TM\right)\hspace{0.1cm}\vert\hspace{0.1cm} \operatorname{div}\left(V\right)\in L^1\left(N\right)\}.
\end{equation*}

Since the admissible competitors in the definition of the Cheeger constant are smooth compact \(n\)-dimensional submanifolds \(N\subset T\left(P,a\right)\) with smooth boundary, it is enough for our purposes to use the classical divergence theorem for smooth vector fields. Below, we shall construct a smooth vector field $V\in L^{\infty}\left(N;TM\right)$, with a finite supremum norm. In this setup, the divergence is naturally locally integrable, i.e. contained in \(L^1\). More details about integrability and validity of the divergence theorem, from the perspective of measure theory, can be found in \cite{evans2018measure}. 

Now, let us take $\mu$ to be the outward unit normal vector of $\partial N$ and some $V\in V_{\operatorname{div}}\left(N\right)\hspace{0.1cm}$, then we can apply the divergence theorem   

\begin{equation*}
\inf_{N}\operatorname{div(V)}\vert N\vert\leq \int_{N}\operatorname{div(V)}=\int_{\partial N}\langle V,\mu\rangle\leq \lVert V \rVert_{\infty} \vert\partial N\vert,
\end{equation*}
which holds for non-zero $V$ and the infimum is considered over all submanifolds with smooth boundary and compact closure.  

By rearranging all terms within the first and the last inequality, we get
\begin{equation}\label{eq:43}
\frac{\vert\partial N\vert}{\vert N\vert}\geq\frac{\operatorname{{inf_N}}\operatorname{div(V)}}{\lVert V \rVert_{\infty}}.
\end{equation}
\textnormal{The right hand side of (\ref{eq:43}) provides a universal lower bound for the \textit{Cheeger constant} $h$, since the infimum in the definition of $h$, runs over all compact admissible smooth Riemannian submanifolds with smooth boundary.}  

\vspace{0.1cm}
In what follows, we will focus on the appropriate parametrization of the curved tube and on the choice of a suitable vector field \(V\), which continues with the computation of its divergence $\operatorname{div}\left(V\right)$.

Let the underlying curve $P$ be parameterized by arc-length $\gamma$, where
\begin{equation}
\gamma:\mathbb{R}/ L\mathbb{Z}\longrightarrow P,
\end{equation}
with \(L=\vert P\vert\), $p=\gamma\left(s\right)$, and write $T\left(s\right)=\dot\gamma\left(s\right)$ for the unit tangent. For \(p=\gamma\left(s\right)\), define the open forward unit hemisphere in \(T_pX\) by
\begin{equation}
\mathbb{S}_{p,+}^{n-1}:=\{\xi\in T_pX\hspace{0.1cm}\vert\hspace{0.1cm}\vert\xi\vert=1,\quad\langle\xi,T\left(s\right)\rangle>0\}.
\end{equation}
For \(a>0\) sufficiently small, the map
\begin{equation}\label{eq.63}
\Psi_a:\bigcup_{p\in P}\mathbb{S}_{p,+}^{n-1}\longrightarrow T^{\circ}\left(P,a\right),\quad \Psi_a\left(p,\xi\right)=\operatorname{exp}_p\left(a\xi\right),
\end{equation}
gives a smooth local parametrization of the interior \(\operatorname{Int}\left(T\left(P,a\right)\right)\) of the curved tube away from a negligible set. In other words, each point of the thin tube can be represented, locally and uniquely, as a point lying on a geodesic sphere of radius \(a\) centered at a suitable point \(p\in P\).

This is the spherical analog of the normal Fermi parametrization
\begin{equation}
\Phi\left(s,\zeta,r\right)=\operatorname{exp}_{\gamma\left(s\right)}\left(r\zeta\right),
\end{equation}
where \(\zeta\in \nu_{\gamma\left(s\right)}P\), \(\vert\zeta\vert =1\). The normal parametrization is adapted to the distance function from the curve \(P\). By contrast, the parametrization \(\Psi_a\) is adapted to geodesic spheres of fixed radius \(a\) centered on the curve. This is precisely what removes the unnecessary algebraic dependence on the curvature of the underlying curve \(P\) within the following arguments. 

However, notice that if a point \(x\in T\left(P,a\right)\) has two different representations \(x=\operatorname{exp}_p\left(a\xi\right)=\operatorname{exp}_q\left(a\widetilde{\xi}\right)\) with \(p\neq q\) and \(\xi\in T_pX,\hspace{0.1cm}\widetilde{\xi}\in T_qX\), then the formula may assign two different values to \(V\left(x\right)\), and the local parametrization would not be enough to define a global vector field \(V\) (such \(V\) would be multivalued). Here, the combination of the compactness of \(P\) with the fact that the radius \(a>0\) can be sufficiently small comes into consideration.

We shall show below that, after decreasing \(a\) if necessary, the moving-sphere parametrization \(\Psi_a\) is valid. Namely, for \(S^{n-1}_{p,+}=\{\xi\in T_pX:\ |\xi|=1,\ \langle \xi,\tau_p\rangle>0\}\), we will show that\(\Psi_a\), given by the relation (\ref{eq.63}), is a smooth diffeomorphism.

This is the curved analog of the elementary parametrization of a straight
Euclidean cylinder. In the model cylinder
\begin{equation}
    \mathcal{C}_a=\{(s,z)\in\mathbb R\times\mathbb R^{n-1}:|z|<a\},
\end{equation}
where every point \(\left(s,z\right)\in\mathcal{C_a}\) has a unique representation
\begin{equation}
    (s,z)=(\sigma,0)+a\xi,
    \qquad |\xi|=1,\quad \xi_1:=\langle\xi, e_1\rangle>0,
\end{equation}
where \(e_1\) is the positively oriented unit tangent vector to the axis \(\mathbb{R}\times\{0\}\). Writing 
\begin{equation}\label{eq.69}
    \xi=\xi_1e_1+\xi^{\perp},\quad \xi^{\perp}\in\{0\}\times\mathbb{R}^{n-1},
\end{equation}
we get
\begin{equation}
\xi^{\perp}=\frac{z}{a},\quad \xi_1=\sqrt{1-\frac{\vert z\vert^2}{a^2}},\quad \sigma=s-a\xi_1.
\end{equation}
Equivalently, one can write the full unit vector as
\begin{equation*}
\xi=\sqrt{1-\frac{\vert z\vert^2}{a^2}}e_1+\frac{z}{a}.
\end{equation*}

For a sufficiently thin curved tube, Fermi coordinates identify the tube with a uniform smooth perturbation of this model. Therefore, the compactness of \(P\) allows the required smallness of \(a\) to be chosen uniformly along the whole underlying curve \(P\). Thus, the above moving-sphere parametrization is really a diffeomorphism after decreasing \(a\), if necessary. This ensures us that for any \(a>0\) sufficiently small, every point \(x\) in the curved tube can be written uniquely in the form \(x=\operatorname{exp}_p\left(a\xi\right)\), where \(p\in P\), \(\vert\xi\vert=1\) and \(\langle\xi, \tau_p\rangle>0\). 

Therefore, the open cylinder is foliated by positively oriented hemispheres of Euclidean spheres of radius \(a\) whose centers lie on the axis. 

Now, let us define the tubular and injectivity radii. The tubular radius is defined as
\begin{equation}
\rho_{tub}\left(P\right):=\sup\{r>0\hspace{0.1cm}\vert\hspace{0.1cm}\operatorname{exp}^{\perp}:\{\left(p,v\right)\in\nu P \hspace{0.1cm}\vert\hspace{0.1cm}\vert\vert v\vert\vert<r\}\longrightarrow \operatorname{Int}\left(T\left(P,r\right)\right)\hspace{0.1cm}\textit{is a diffeomorphism} \}.
\end{equation}
This is an extrinsic quantity associated with the underlying curve \(P\), and it basically indicates how far we can move away from \(P\) before two different normal geodesics meet. Naturally, we have \(0<a<\rho_{tub}\left(P\right)\), so every point of the tube has a unique nearest point on \(P\). If \(a>\rho_{tub}\left(P\right)\), then two different normal geodesics may intersect, and the curved tube \(T\left(P,a\right)\) would self-intersect, which we avoid.

The injectivity radius is an intrinsic property of the ambient manifold. Here, we have to approach its definition from a more general perspective. Let \(M\) be a smooth Riemannian manifold. Then, for a point \(p\in M\), it is defined as
\begin{equation}
\rho_{inj}\left(p\right):=\sup\{r>0\hspace{0.1cm}\vert\hspace{0.1cm}\operatorname{exp}_p\hspace{0.1cm}is\hspace{0.1cm}injective\hspace{0.1cm}on\hspace{0.1cm}B_r\left(0\right)\subset T_pM\}.
\end{equation}
For the entire manifold \(M\), we can define the injectivity radius \(\rho_{inj}^M\left(M\right)\) by taking the infimum of \(\rho_{inj}\left(p\right)\) over all points \(p\in M\). In particular, this means that we are considering the ambient injectivity radius along \(P\subset M\).

When \(M=X\), the injectivity radius \(\rho_{inj}^X\left(M\right)\) simply becomes \(\rho_{inj}\left(X\right)\). In the Euclidean and hyperbolic cases, the exponential map is globally injective, so \(\rho_{inj}=\infty\) when \(K^X=0\) and \(K^X<0\). However, in the spherical case, every geodesic reaches the antipodal point. Therefore, the exponential map stops being injective there. If the sectional curvature is \(K^X>0\), then the injectivity radius satisfies \(\rho_{inj}\left(X\right)=\frac{\pi}{\sqrt{K^X}}\). Both radii control different phenomena. The injectivity radius controls whether geodesics from one fixed point remain unique. The tubular radius controls whether normal geodesics from different points of the curve remain disjoint. 

Thus, \(\rho_{tub}\left(P\right)\neq\rho_{inj}\left(X\right)\). Additionally, suppose that \(r<\rho_{tub}\left(P\right)\). Then the normal exponential map is injective on the radius-\(r\) normal bundle. Fix one point \(p\in P\). Restricting the normal exponential map to the fiber \(\nu_p P\subset T_pX\) still gives an injective map. But the restriction is simply the exponential map \(\operatorname{exp}_p\) restricted to the normal directions. Therefore, if \(r>\rho_{inj}\left(X\right)\), the exponential map at \(p\) has already lost injectivity before radius \(r\), so the normal exponential map cannot remain injective. Therefore, one always has
\begin{equation}
\rho_{tub}\left(P\right)\leq\rho_{inj}\left(X\right).
\end{equation}
However, this inequality is usually not sharp. The reason is that the tubular radius is often determined long before the ambient cut locus appears. As we have already mentioned, \(\rho_{inj}\left(\mathbb{R}^n\right)=\infty\), while for a closed curve, we have \(\rho_{tub}\left(P\right)<\infty\). In the spherical case, a small embedded circle has a tubular radius roughly equal to its focal radius or self-distance, both strictly smaller than \(\frac{\pi}{\sqrt{K^X}}\), so \(\rho_{tub}\left(P\right)<\rho_{inj}\left(X\right)\), when \(K^X>0\). Equality occurs, for example, when \(P\) is a complete geodesic in the Euclidean space, then \(\rho_{tub}\left(\mathbb{R}^n\right)=\rho_{inj}\left(\mathbb{R}^n\right)=\infty\). Generally speaking, equality occurs whenever the first obstruction to the injectivity of the normal exponential map comes from the ambient cut locus rather than from local points or intersections of distinct normal geodesics.  

In our case, since the normal exponential map is the restriction of the ambient exponential map to the normal bundle, the tubular radius always satisfies \(\rho_{tub}\left(P\right)\leq\rho_{inj}\left(X\right)\). Consequently, if \(a<\rho_{tub}\left(P\right)\), then automatically \(a<\rho_{inj}\left(X\right)\). 

The following proposition shows that the same description remains valid for the open curved tube, provided the radius is below the tubular radius and the ambient injectivity radius along \(P\). Under such circumstances for \(a\), the moving-sphere parametrization is a smooth diffeomorphism onto the interior of \(T\left(P,a\right)\) and does not develop an interior double point.

\newpage
\begin{prop}\label{prop.4}
Let \(P\subset X\) be a smooth embedded compact oriented curve. Let \(\gamma:\mathbb{R}/L\mathbb{Z}\longrightarrow X\) be a unit-speed parametrization of \(P\), and by \(T\left(s\right)=\dot{\gamma}\left(s\right)\) we denote the positively oriented unit tangent vector to \(P\). Let \(\rho_{tub}\left(P\right)>0\) be a tubular radius of \(P\) such that, for every \(0<a<\rho_{tub}\left(P\right)\), the normal exponential map 
\begin{equation*}
\Phi\left(s,v\right)=\operatorname{exp}_{\gamma\left(s\right)}\left(v\right),\quad v\in \nu_{\gamma\left(s\right)}P,\hspace{0.2cm}\vert v\vert<a
\end{equation*}
is a smooth diffeomorphism from the open disk bundle of radius \(a\) onto the open tube \(T^{\circ}\left(P,a\right)\). Next, assume that \(0<a<\min\{\rho_{tub}\left(P\right),\rho_{inj}\left(X\right)\}\). Then the moving-sphere parametrization
\begin{equation}
\Psi_a:\mathcal{H}_{+}:=\bigcup_{p\in P}\mathbb{S}^{n-1}_{p,+}\longrightarrow T^{\circ}\left(P,a\right),\quad \Psi_a\left(p,\xi\right)=\operatorname{exp}_p\left(a\xi\right),
\end{equation}
is a smooth diffeomorphism.
\end{prop}

\hspace{-0.5cm}\textbf{Proof of proposition \ref{prop.4}.} We divide the following proof into several steps. First, we show that the moving-sphere parametrization sends the positive hemisphere bundle into the open tube. 

Let \(p\in P\) and \(\tau_p\in T_pP\) be the positively oriented unit tangent vector to \(P\) at \(p\). Next, let \(\xi\in T_pX\), with \(\vert\xi\vert=1\) and \(\langle\xi,\tau_p\rangle>0\). Set \(x=\operatorname{exp}_p\left(a\xi\right)\). Since \(a<\rho_{inj}\left(P\right)\), the geodesic segment 
\begin{equation}
r\mapsto\operatorname{exp}_p\left(r\xi\right),\quad 0\leq r\leq a,
\end{equation}
is minimizing and depends smoothly on the initial vector. Hence, \(d\left(p,x\right)=a\). 

Now, let \(\gamma\left(t\right)\) be a unit-speed parametrization of \(P\) near the point \(p\), chosen so that
\begin{equation}
\gamma\left(0\right)=p,\quad\dot{\gamma}\left(0\right)=\tau_p.
\end{equation}
Consider the function
\begin{equation}
F\left(t\right)=\frac{1}{2}d\left(\gamma\left(t\right),x\right)^2.
\end{equation}
Using the first variation formula for the squared distance, we obtain
\begin{equation}
F^{\prime}\left(0\right)=-a\langle\xi,\tau_p\rangle.
\end{equation}
Since \(\langle\xi,\tau_p\rangle>0\), we get
\begin{equation}
F^{\prime}\left(0\right)<0.
\end{equation}
Therefore, for all sufficiently small positive \(t\), we have
\begin{equation}
d\left(\gamma\left(t\right),x\right)<a.
\end{equation}
Because \(\gamma\left(t\right)\in P\), this implies
\begin{equation}
\operatorname{dist}\left(x,P\right)\leq d\left(\gamma\left(t\right),x\right)<a.
\end{equation}
Thus \(x\in T^{\circ}\left(P,a\right)\). Consequently, we have \(\Psi_a\left(\mathcal{H}_{+}\right)\subset T^{\circ}\left(P,a\right)\).

Now, we focus on proving the surjectivity and uniqueness. 

Let us fix an arbitrary point \(q\in T^{\circ}\left(P,a\right)\). Since \(a<\rho_{tub}\left(P\right)\), then every such point has a unique Fermi representation
\begin{equation}\label{eq.82}
q=\operatorname{exp}_{\gamma\left(s_0\right)}\left(v_0\right),\quad v_0\in \nu_{\gamma\left(s_0\right)}P,\hspace{0.2cm}\vert v_0\vert=r<a.
\end{equation}
Equivalently, \(\gamma\left(s_0\right)\) is the unique nearest point of \(P\) to \(q\), and \(\operatorname{dist}\left(q,P\right)=r<a\).

Now, define the squared distance function along \(P\) by 
\begin{equation}
F_q\left(s\right)=d\left(q,\gamma\left(s\right)\right)^2.
\end{equation}
Since \(a<\rho_{inj}\left(P\right)\), the function \(F_q\) is smooth at every point \(s\) satisfying
\begin{equation}
d\left(q,\gamma\left(s\right)\right)\leq a.
\end{equation}
Consider the open sublevel set \(A_q=\{s\in\mathbb{R}/L\mathbb{Z}\hspace{0.1cm}\vert\hspace{0.1cm}F_q\left(s\right)<a^2\}\). This set is nonempty because it contains the point \(s_0\). Moreover, since \(P\) is connected embedded closed curve and we assume that \(a<\rho_{tub}\left(P\right)\), the set \(A_q\) is necessarily a single connected open arc of the oriented circle \(\mathbb{R}/L\mathbb{Z}\). Geometrically, \(A_q\) records those centers whose geodesic balls of radius \(a\) contain the fixed point \(q\). Because the tube is embedded, the curve cannot leave this family of balls and re-enter it elsewhere without violating the uniqueness of the tubular neighborhood. 

Now, let us denote the two boundary points of \(A_q\) by \(s_-\) and \(s_+\), ordered so that \(A_q\) is entered at \(s_-\) and exited at \(s_+\) with respect to the positive orientation of \(P\).

At both points, we have
\begin{equation}
F_q\left(s_{\pm}\right)=a^2,
\end{equation}
or, equivalently, 
\begin{equation}
d\left(q,\gamma\left(s_{\pm}\right)\right)=a.
\end{equation}

At this stage, we move to proving that the intersections with the level set \(F_q=a^2\) are transverse. 

Suppose that for one of these boundary points, say \(s_*\), we have
\begin{equation}
F^{\prime}_q\left(s_*\right)=0.
\end{equation}
Then, by the first variation formula, the geodesic from \(\gamma\left
(s_*\right)\) to \(q\) is orthogonal to \(P\) at \(\gamma\left(s_*\right)\). Since 
\begin{equation}
d\left(q,\gamma\left(s_*\right)\right)=a
\end{equation}
and \(a<\rho_{tub}\left(P\right)\), this gives us a normal representation of \(q\) by a normal vector of length \(a\). On the other hand, \(q\) already has a Fermi representation \(q=\operatorname{exp}_{\gamma\left(s_0\right)}\left(v_0\right)\), with \(\vert v_0\vert=r<a\). These two normal representations are distinct, because \(s_0\) lies in the interior of \(A_q\), whereas the point \(s_*\) lies on its boundary. This contradicts injectivity of the normal exponential map on the normal disk bundle of the radius \(\rho_{tub}\left(P\right
)\). Therefore,
\begin{equation}
F^{\prime}_q\left(s_{\pm}\right)\neq 0.
\end{equation}
Since \(A_q\) is entered at \(s_-\), we have 
\begin{equation}
F^{\prime}_q\left(s_-\right)<0,
\end{equation}
and since \(A_q\) is exited at \(s_+\), we have
\begin{equation}
F^{\prime}_q\left(s_+\right)>0.
\end{equation}

For each boundary point \(s\in\{s_-,s_+\}\), let \(\xi\left(s\right)\in T_{\gamma\left(s\right)}X\) be the initial unit vector of the unique minimizing geodesic from \(\gamma\left(s\right)\) to \(q\). Thus \(q=\operatorname{exp}_{\gamma\left(s\right)}\left(a\xi\left(s\right)\right)\).
The first variation formula yields
\begin{equation}
F^{\prime}_q\left(s\right)=-2a\langle\xi\left(s\right),T\left(s\right)\rangle.
\end{equation}
Therefore,
\begin{equation}
F^{\prime}_q\left(s\right)<0\iff\langle\xi\left(s\right),T\left(s\right)\rangle>0,
\end{equation}
and
\begin{equation}\label{eq.94}
F^{\prime}_q\left(s\right)>0\iff\langle\xi\left(s\right),T\left(s\right)\rangle<0.
\end{equation}
Consequently, among the two boundary points \(s_-\) and \(s_+\), exactly one gives a positively oriented direction. Namely, 
\begin{equation}
\langle\xi\left(s_-\right),T\left(s_-\right)\rangle>0,
\end{equation}
whereas
\begin{equation}
\langle\xi\left(s_+\right),T\left(s_+\right)\rangle<0.
\end{equation}
This implies that every point \(q\in T^{\circ}\left(P,a\right)\) has at least one positively oriented moving-sphere representation
\begin{equation}
q=\operatorname{exp}_{\gamma\left(s_-\right)}\left(a\xi\left(s_-\right)\right),\quad \vert\xi\left(s_-\right)\vert=1, \hspace{0.2cm}\langle\xi\left(s_-\right),T\left(s_-\right)\rangle>0.
\end{equation}

Now, we shall proceed to prove the uniqueness. Suppose that \(q\) has another positively oriented representation
\begin{equation}
q=\operatorname{exp}_{\gamma\left(\sigma
\right)}\left(a\widetilde{\xi}\right),\quad \vert\widetilde{\xi}\vert=1,\hspace{0.2cm}\langle\widetilde{\xi},T\left(\sigma\right)\rangle>0.
\end{equation}
Then 
\begin{equation}
F_q\left(\sigma\right)=a^2.
\end{equation}
Moreover, by using the first variation formula, we obtain
\begin{equation}
F^{\prime}_q\left(\sigma\right)=-2a\langle\widetilde{\xi}, T\left(\sigma\right)\rangle<0.
\end{equation}
Therefore, \(\sigma\) is an entering boundary point of the sublevel set \(A_q\). However, the set \(A_q\) is a single connected open interval within the oriented circle, and therefore, it has exactly one entering boundary point, namely \(s_-\). Hence,
\begin{equation}
\sigma=s_-.
\end{equation}
Since \(a<\rho_{inj}\left(P\right)\), the minimizing geodesic from \(\gamma\left(s_-\right)\) to \(q\) is unique. Therefore,
\begin{equation}
\widetilde{\xi}=\xi\left(s_-\right).
\end{equation}
This proves the uniqueness of the positively oriented moving-sphere parametrization, so
\begin{equation}
\Psi_a:\mathcal{H}_{+}\longrightarrow T^{\circ}\left(P,a\right)
\end{equation}
is bijective.

The last remaining step consists of proving the smoothness of the inverse. 

Let \(q\in T^{\circ}\left(P,a\right)\), and let \(s\left(q\right)\) denote the unique positively oriented center constructed in the previous step. It is characterized by
\begin{equation}
F_q\left(s\left(q\right)\right)=a^2,\quad F^{\prime}_q\left(s\left(q\right)\right)<0.
\end{equation}
At this point, we have
\begin{equation}
\restr{\frac{\partial}{\partial s}\left(F_q\left(s\right)-a^2\right)}{s=s\left(q\right)}=F^{\prime}_q\left(s\left(q\right)\right)\neq 0.
\end{equation}
Applying the implicit function theorem, the positively oriented center \(s\left(q\right)\) locally depends smoothly on \(q\). Since the positively oriented center is unique, these local smooth functions agree on overlaps. Thus \(q\mapsto s\left(q\right)\) is globally smooth. 

Furthermore, because \(a<\rho_{inj}\left(P\right)\), the inverse exponential map
\begin{equation}
\operatorname{exp}_{\gamma\left(s\left(q\right)\right)}^{-1}\left(q\right)
\end{equation}
is smooth for all relevant pairs \(\left(\gamma\left(s\left(q\right)\right),q\right)\). Therefore,
\begin{equation}
\xi\left(q\right)=\frac{1}{a}\operatorname{exp}_{\gamma\left(s\left(q\right)\right)}^{-1}\left(q\right)
\end{equation}
depends smoothly on \(q\).
This implies that the inverse map
\begin{equation}
q\mapsto\left(\gamma\left(s\left(q\right)\right),\xi\left(q\right)\right)
\end{equation}
is smooth. 

In summary, the map \(\Psi_a\) itself is smooth because it is the restriction of the smooth exponential map. Since \(\Psi_a\) is bijective and has a smooth inverse, it is directly a smooth diffeomorphism. 

This completes the proof of proposition \ref{prop.4}.

\vspace{0.2cm}
Equivalently, proposition \ref{prop.4} states that for every point \(x\in T^{\circ}\left(P,a\right)\), there exists a unique pair \(\left(p,\xi\right)\in\mathcal{H}_{+}\) such that
\begin{equation}
x=\operatorname{exp}_p\left(a\xi\right),\quad \vert\xi\vert=1,\hspace{0.2cm}\langle\xi, \tau_p\rangle>0.
\end{equation}
Moreover, this pair depends smoothly on \(x\).

Notice that it is important to distinguish between the open tube and the closed tube in the following sense. The ordinary Fermi parametrization describes the open tube via normal disks. A point of the open tube is written as \(x=\operatorname{exp}_{\gamma\left(s\right)}\left(v\right)\), where \(v\in \nu_{\gamma\left(s\right)}P\) and \(\vert v\vert<a\). Thus, the Fermi parametrization uses the normal disk of radius \(a\) over each point of the underlying curve \(P\). 

On the other hand, the moving-sphere parametrization describes the same open tube by positively oriented open hemispheres, with a point being written as \(x=\operatorname{exp}_p\left(a\xi\right)\), where \(p\in P\), \(\vert\xi\vert=1\) and \(\langle\xi,\tau_p\rangle>0\). It is clear that these two parametrizations use different geometric slices of the same object. The Fermi slices are normal disks, while the moving-sphere slices are positively oriented hemispheres of geodesic spheres of radius \(a\) centered at points of \(P\). The crucial point is that, for the open curved tube, the moving-sphere parametrization cannot develop an interior double point before the Fermi projection ceases to be unique. 

In order to justify this, assume that the normal exponential map is injective on the whole open normal disk bundle of radius \(a\). Equivalently, every point of the open tube has a unique Fermi representation. When we examine the considerations between the relations (\ref{eq.82}) and (\ref{eq.94}), it is clear that an interior double point of the moving-sphere parametrization would potentially force a second Fermi projection. However, precisely because of the considerations between (\ref{eq.82}) and (\ref{eq.94}), this does not happen. A second Fermi projection is excluded by the injectivity of the normal exponential map. Therefore, for the open curved tube, the positively oriented moving-sphere parametrization cannot fail before the Fermi parametrization fails. 

The situation is different for the closed tube. In the Fermi parametrization, the boundary of the curved tube is given by \(\vert v\vert=a\). On the other hand, in the case of the moving-sphere parametrization, the boundary corresponds to the equator of the moving sphere, namely to the condition \(\langle\xi,\tau_p\rangle=0\). Thus, the open hemisphere condition \(\langle\xi, \tau_p\rangle>0\) degenerates exactly at the boundary. 

This degeneracy already appears in the straight Euclidean cylinder. In that model, the transition from moving-sphere coordinates to Fermi coordinates has the form 
\begin{equation}
z=a\xi^{\perp},
\end{equation}
and
\begin{equation}
s=\sigma+a\sqrt{1-\vert\xi^{\perp}\vert^2}, \quad \vert\xi^{\perp}\vert<1.
\end{equation}
The condition \(\vert\xi^{\perp}\vert<1\) corresponds to the open hemisphere. The inverse transformation is given by
\begin{equation}
\xi^{\perp}=\frac{z}{a},
\end{equation}
and
\begin{equation}
\sigma=s-a\sqrt{1-\frac{\vert z\vert^2}{a^2}},
\end{equation}
where \(\xi^{\perp}\) can be understood as the normal component of the unit vector \(\xi\) that is given by (\ref{eq.69}). This inverse is smooth for all \(\vert z\vert<a\), which is precisely the open curved tube. However, as \(\vert z\vert\rightarrow a\), the term \(\sqrt{1-\frac{\vert z\vert^2}{a^2}}\) loses smoothness. Its derivative becomes singular at \(\vert z\vert=a\).

Geometrically, this is the fold at the equator of the moving hemisphere. The equator corresponds exactly to the boundary of the Fermi disk. 

Therefore, the open tube admits a genuine diffeomorphic relation between the Fermi disk parametrization and the positively oriented moving-hemisphere parametrization. The closed tube does not have the same smooth diffeomorphic relation, because the boundary (equator) produces a degeneracy of the inverse. 

In summary, we have the following cases:
\begin{itemize}
    \item \underline{Open tube:} no earlier interior failure of the hemisphere parametrization
    \item \underline{Closed tube:} boundary degeneracy occurs at the equator of the moving hemisphere
\end{itemize}
Therefore, the moving-sphere parametrization appears consistent on the open tube, and its unavoidable degeneracy emerges only when one tries to include the boundary.

\vspace{0.2cm}
Based on proposition \ref{prop.4}, the following two things are guaranteed:
\begin{itemize}
    \item \(V\) is well defined, because any form of \(V\left(x\right)\) does not depend on a choice of \(p\) or \(\xi\),
    \item \(V\) is smooth inside the curved tube, so the map \(\Psi_a\left(p,\xi\right)=\operatorname{exp}_p\left(a\xi\right)\) is a diffeomorphism, and \(\left(p,\xi\right)\) depends smoothly on \(x\).
\end{itemize}

Now we can define the vector field \(V\) explicitly. If
\begin{equation}
x=\Psi_a\left(p,\xi\right)=\operatorname{exp}_p\left(a\xi\right),
\end{equation}
then we can set
\begin{equation}\label{eq.106}
V\left(x\right):=\restr{\frac{d}{dr}}{r=a}\operatorname{exp}_p\left(r\xi\right).
\end{equation}
Equivalently, the vector field \(V\left(x\right)\) is the outward radial unit vector at \(x\) for the geodesic sphere of radius \(a\) centered at \(p\in P\).

Since the curve \(r\mapsto\operatorname{exp}_p\left(r\xi\right)\) is parameterized by arc-length, we have
\begin{equation}
\vert V\left(x\right)\vert=1,
\end{equation}
whenever \(V\) is defined smoothly. Therefore, \(\left\| V\right\|_{\infty}=1\). Now, it remains to compute the divergence \(\operatorname{div}\left(V\right)\). Fix the point \(p\in P\) and let us consider the hypersurface 
\begin{equation}
\Sigma_p:=\{\operatorname{exp}_p\left(a\xi\right)\hspace{0.1cm}\vert\hspace{0.1cm}\xi\in\mathbb{S}_{p,+}^{n-1}\}.
\end{equation}
This represents an open piece of the geodesic sphere of radius \(a\) centered at \(p\). Along \(\Sigma_p\), the vector field \(V\) is exactly the outward radial unit normal to that geodesic sphere. 

Let \(x\in\Sigma_p\) and choose an orthonormal basis \(\{e_1,...e_{n-1}\}\) of the tangent space \(T_x\Sigma_p\). We can complete it to an orthonormal basis of \(T_x X\) by adding the element \(e_n=V\left(x\right)\). Then the divergence is given by
\begin{equation}\label{eq.52}
\operatorname{div}\left(V\right)=\sum_{j=1}^{n-1}\langle \nabla_{e_j}V,e_j\rangle+\langle\nabla_VV,V\rangle.
\end{equation}
Since the vector field \(V\) has unit length, we have \(\vert V\vert^2=\langle V,V\rangle=1\). Differentiating this identity in the direction of any arbitrary smooth vector field \(\widetilde{V}\) yields
\begin{equation}
0=\widetilde{V}\left(\vert V\vert^2\right)=\widetilde{V}\langle V,V\rangle=\langle\nabla_{\widetilde{V}}V,V\rangle+\langle V,\nabla_{\widetilde{V}}V\rangle=2\langle\nabla_{\widetilde{V}}V,V\rangle.
\end{equation}
Therefore, \(\langle\nabla_{\widetilde{V}}V,V\rangle=0\) for any \(\widetilde{V}\). In particular, taking \(\widetilde{V}=V\) gives us
\begin{equation}
\langle \nabla_VV,V\rangle=0.
\end{equation}
Thus, the \(V\)-direction does not contribute to the divergence. 

Plugging this back into the divergence relation (\ref{eq.52}), we obtain
\begin{equation}
\operatorname{div}\left(V\right)=\sum_{j=1}^{n-1}\langle \nabla_{e_j}V,e_j\rangle,
\end{equation}
so just the first term remains.

For technical purposes, we define the functions $S\left(r\right)$ and $C\left(r\right)$ as solutions to the following differential equations
\begin{equation}\label{eq:45}
S^{''}\left(r\right)+K^XS\left
(r\right)=0,
\end{equation}
\begin{equation}\label{eq:46}
C^{''}\left(r\right)+K^XC\left
(r\right)=0,
\end{equation}
with the different initial conditions $S\left(0\right)=0$, $S{'}\left(0\right)=1$ and $C\left(0\right)=1$, $C{'}\left(0\right)=0$. The explicit form of $S\left(r\right)$ is given by 
\begin{equation*}
S\left(r\right)=
\begin{cases}
\frac{1}{\sqrt{K^X}}\operatorname{sin}\left(\sqrt{K^X}r\right), \hspace{0.2cm}K^X>0 \\[1ex]
r, \hspace{0.2cm}K^X=0 \\[1ex]
\frac{1}{\sqrt{\vert K^X\vert}}\operatorname{sinh}\left(\sqrt{\vert K^X\vert}r\right), \hspace{0.2cm}K^X<0.
\end{cases}
\end{equation*}
Notice that $C\left(r\right)=S{'}\left(r\right)$. In literature focused on Riemannian geometry, the functions $S\left(r\right)$ and $C\left(r\right)$ are often called the model functions \cite{petersen2006riemannian}. 

In order to proceed further, we need the following lemma.

\begin{lemma}\label{lem.7}
\textit{Considering the \(n\)-dimensional real space form \(X\) of constant curvature \(K^X\), let us fix \(p\in P\) and a unit vector \(\xi\in T_pX\). Let \(\alpha\left(r\right)=\operatorname{exp}_p\left(r\xi\right)\) be the unit speed geodesic starting at \(p\) in the direction \(\xi\), where \(0<r\leq a\). Next, let \(E\in T_pX\) satisfy \(E\perp\xi\). Denote by \(E\left(r\right)\) the parallel transport of \(E\) along \(\alpha\). Then the Jacobi field \(J\) along \(\alpha\), subject to the initial conditions \(J\left(0\right)=0,\hspace{0.1cm}  \)\(J^{\prime}\left(0\right)=E\), is given by}
\begin{equation}
J\left(r\right)=S\left(r\right)E\left(r\right),
\end{equation}
satisfying
\begin{equation}
J^{\prime}\left(r\right)=\frac{C\left(r\right)}{S\left(r\right)}J\left(r\right).
\end{equation}
\end{lemma}

\hspace{-0.6cm}\textbf{Proof of lemma \ref{lem.7}.}
Since \(X\) has a constant sectional curvature \(K^X\), its curvature tensor satisfies
\begin{equation}
R\left(Y,\alpha{'}\right)\alpha^{'}=K^XY,
\end{equation}
whenever \(Y\perp\alpha{'}\). Therefore, every normal Jacobi field \(J\perp\alpha^{\prime}\) satisfies the Jacobi equation
\begin{equation}\label{eq.61}
J^{''}+K^XJ=0.
\end{equation}
Consider the radially dependent solution \(J\left(r\right)=f\left(r\right)E\left(r\right)\), where \(E\left(r\right)\) is parallel along \(\alpha\) and \(f\left(r\right)\) is a smooth radial function. Since \(E{'}\left(r\right)=0\), the Jacobi equation (\ref{eq.61}) reduces to the scalar equation 
\begin{equation}
f^{\prime\prime}\left(r\right)+K^Xf\left(r\right)=0.
\end{equation}
The initial conditions become \(f\left(0\right)=0,\quad f{'}\left(0\right)=1\). By the definition of the model function \(S\), the unique solution is
\begin{equation}
f\left(r\right)=S\left(r\right).
\end{equation}
Hence,
\begin{equation}
J\left(r\right)=S\left(r\right)E\left(r\right).
\end{equation}

This completes the proof of lemma \ref{lem.7}.

\vspace{0.2cm}
We can directly use the above result for the computation of the divergence. Let \(\alpha\left(r\right)=\operatorname{exp}_p\left(r\xi\right)\) be the radial geodesic from \(p\in P\), and let \(\partial_r=\alpha^{\prime}\left(r\right)\) be the outward radial unit vector field. Now, vary the initial direction \(\xi\) by a curve \(\xi\left(t\right)\) on the unit sphere in \(T_pX\). This produces a geodesic variation
\begin{equation}
F\left(r,t\right)=\operatorname{exp}_p\left(r\xi\left(t\right)\right).
\end{equation}
The two natural vector fields along this variation are
\begin{equation}\label{eq.136}
\partial_r=\frac{\partial F}{\partial r},\quad J\left(r\right)=\frac{\partial F}{\partial t}.
\end{equation}
Since the Levi-Civita connection is torsion-free, mixed covariant derivatives commute \(\nabla_{\partial t}\hspace{0.1cm}\partial r=\nabla_{\partial r}\hspace{0.1cm}\partial t\).
From (\ref{eq.136}), we know that \(\partial_t=J\left(r\right)\). This gives us
\begin{equation}
\nabla_{J\left(r\right)}\hspace{0.1cm}\partial r=\nabla_{\partial_r}J\left(r\right).
\end{equation}
But \(\nabla_{J\left(r\right)}\hspace{0.1cm}\partial r\) is precisely the covariant derivative of \(J\) along the radial geodesic, which is denoted by \(J^{\prime}\left(r\right)\). Finally, using the explicit formula for \(J\), we can directly see that 
\begin{equation}
J^{\prime}\left(r\right)=C\left(r\right)E\left(r\right)=\frac{C\left(r\right)}{S\left(r\right)}J\left(r\right).
\end{equation}

\vspace{0.2cm}
Now, we set \(r=a\) in the above formula, with \(V=\partial_r\) being our calibrating vector field, where \(\partial_r\) corresponds to the outward radial unit field of the geodesic sphere of radius \(a\). If \(e_j\) is any tangent vector to such a geodesic sphere, then it can be represented as some value of a Jacobi field \(e_j=J\left(a\right)\). Therefore,
\begin{equation}
\nabla_{e_j}V=\nabla_{J\left(a\right)}\partial_r=\frac{C\left(a\right)}{S\left(a\right)}J\left(a\right)=\frac{C\left(a\right)}{S\left(a\right)}e_j.
\end{equation}
This implies that 
\begin{equation}
\langle\nabla_{e_j}V,e_j\rangle=\frac{C\left(a\right)}{S\left(a\right)},
\end{equation}
for every \(j\in\{1,...,n-1\}\). Taking the trace over an orthonormal basis of the tangent space to the geodesic sphere gives us
\begin{equation}\label{eq.143}
\operatorname{div}\left(V\right)=\left(n-1\right)\frac{C\left(a\right)}{S\left(a\right)},
\end{equation}
almost everywhere in \(T\left(P,a\right)\). 

In relation to our previous considerations, the representation \(x=\exp_p(a\xi)\), with \(\langle \xi,\tau_p\rangle>0\) is unique and depends smoothly on \(x\) inside the curved tube. Hence, the vector field \(V\) defined by the relation (\ref{eq.106}) is a globally well-defined smooth vector field on the interior of \(T\left(P,a\right)\). Therefore, the preceding pointwise computation of \(\operatorname{div}V\) holds everywhere inside the tube. This makes the preceding computation valid in the weak sense, so \(V\) is admissible in the class \(V_{div}\). 

Moreover, in the spherical case \(K^X>0\), the geometric admissibility assumption on the curved tube already implies that \(a<\frac{\pi}{2\sqrt{K^X}}\), as observed in the paragraph right below theorem \ref{thm:1}. Therefore, 
\begin{equation}
\frac{C\left(a\right)}{S\left(a\right)}=\sqrt{K^X}\operatorname{cotg}\left(\sqrt{K^X}a\right)>0,
\end{equation}
and the hemisphere construction remains valid.

Now, substituting (\ref{eq.143}) and \(\vert\vert V\vert\vert_{\infty}\) into the estimate (\ref{eq:43}), we obtain
\begin{equation}
\frac{\vert\partial N\vert}{\vert N\vert}\geq\left(n-1\right)\frac{C\left(a\right)}{S\left(a\right)},
\end{equation}
for every smooth admissible \(N\subset T\left(P,a\right)\). Taking the infimum over all such \(N\) gives us
\begin{equation}
h\left(T\left(P,a\right)\right)\geq\left(n-1\right)\frac{C\left(a\right)}{S\left(a\right)}.
\end{equation}

Finally, substituting the explicit formulas for \(S\) and \(C\), we get
\begin{equation}
      h\left(T\left(P,a\right)\right)\geq\begin{cases}
\vspace{0.25cm}
\left(n-1\right)\sqrt{K^X}\operatorname{cotg}\left(\sqrt{K^X}a\right), & K^X>0 \\
\vspace{0.25cm}  \frac{n-1}{a}, & K^X=0 \\
\left(n-1\right)\sqrt{\vert K^X\vert}\operatorname{cotgh}\left(\sqrt{\vert K^X\vert}a\right), & K^X<0.
    \end{cases} 
\end{equation}

Combining this lower bound with the upper bound obtained in section \ref{sec.3.1} gives us
\begin{equation}
h\left(T\left(P,a\right)\right)=\begin{cases}
\vspace{0.25cm}
\left(n-1\right)\sqrt{K^X}\operatorname{cotg}\left(\sqrt{K^X}a\right), & K^X>0 \\
\vspace{0.25cm}
\frac{n-1}{a}, & K^X=0 \\
\left(n-1\right)\sqrt{\vert K^X\vert}\operatorname{cotgh}\left(\sqrt{\vert K^X\vert}a\right), & K^X<0.
    \end{cases} 
\end{equation}

Moreover, since the upper bound is attained by the full tube \(T\left(P,a\right)\), the tube is self-Cheeger, i.e. \(C_{T\left(P,a\right)}=T\left(P,a\right)\). This completes the proof of theorem \ref{thm:1}.

\newpage
\newtheorem{rem}{Remark}
\begin{rem}

\vspace{0.1cm}
It is worth noting that the lower bound vector field used above is intentionally not the raw field \(\nabla r\), where \(r\left(x\right)=\operatorname{dist}\left(x,P\right)\) with \(x\in T\left(P,a\right)\setminus P\). This is very important, because the normal tubular hypersurface \(\delta P_r\) has a full Jacobian of the form
\begin{equation}
J\left(s,\zeta,r\right)=S\left(r\right)^{n-2}\left(C\left(r\right)-\kappa_{\zeta}\left(s\right)S\left(r\right)\right),
\end{equation}
where \(\kappa_{\zeta}\left(s\right)=\langle\nabla_TT,\zeta\rangle\). Thus, the geometry of \(\delta P_r\) still contains a longitudinal contribution depending on the curvature of the core curve \(P\). This term cancels only after integration over the unit normal sphere, because \(\zeta\mapsto\kappa_{\zeta}\left(s\right)\) is linear and hence has a zero spherical average.

That cancellation is sufficient for the upper-bound computation. However, the lower-bound argument requires a pointwise divergence estimate. The moving-sphere field entirely avoids this difficulty because its divergence is computed from the geodesic sphere of radius \(a\), whose mean curvature in a space form is exactly
\begin{equation}
\left(n-1\right)\frac{C\left(a\right)}{S\left(a\right)}.
\end{equation}

Thus, the proof remains consistent with the upper-bound calculation while avoiding the unnecessary and fragile pointwise computation of the Laplacian of the distance function \(r\).
\end{rem}

\section{Unbounded curved tubes}

\vspace{0.2cm}
At last, it is interesting to explore the case of unbounded/noncompact curved tubes. The compactness assumption on the core curve \(P\) is essential for the self-Cheeger conclusion. In the compact case, the whole tube \(T\left(P,a\right)\) has finite volume and finite boundary area, so it is admissible as a Cheeger competitor. This allows one to test the Cheeger constant on the whole tube and then prove, via the divergence argument, that the tube itself realizes the infimum.

The situation completely changes when \(P\) is noncompact and has infinite length. Hence, suppose that \(P\) is a smooth embedded noncompact curve in \(X^n\). Then we have an embedded unbounded tube, which we can denote by \(T_{\infty}\left(P,a\right)\). It is crucial to notice that the noncompact uniformly tubular setting can occur only for \(K^X\leq 0\). Indeed, when \(K^X>0\), the ambient space \(X^n\) is a sphere, thus, it is automatically compact. If \(\gamma:\mathbb{R}\longrightarrow X^n\) is an arc-length parametrization of a noncompact curve, then the sequence \(x_j=\gamma\left(j\right)\) possesses a convergent subsequence \(x_{j_k}\longrightarrow x_{\infty}\). Hence, this subsequence is Cauchy, so \(d_{X^n}\left(x_{j_k},x_{j_l}\right)\longrightarrow 0\), for \(k,l\longrightarrow\infty\). Consequently, portions of the curve with arbitrarily large intrinsic separation become arbitrarily close in the ambient space. Therefore, no positive uniform tubular radius can exist along the whole curve. In particular, the unbounded tubular setting discussed in this section has no spherical counterpart and shall be considered only for \(K^X\leq 0\).

Let \(\gamma:\mathbb{R}\longrightarrow P\) be an arc-length parametrization of the core curve, and let \(P_L=\gamma\left([-\frac{L}{2},\frac{L}{2}]\right)\) be the compact subarc of length \(L\). Next, let \(T_L\left(P,a\right)\) be the corresponding finite tube over \(P_L\). This finite tube has three boundary 
parts: the lateral tubular boundary and two end caps. The lateral boundary grows linearly with \(L\), while the two end caps have a size that is independent of \(L\). Additionally, notice that the codimension-two corners where the lateral boundary meets the end caps may be smoothed with an arbitrarily small change of volume and boundary area. Hence, these truncations yield admissible smooth competitors after approximation. 

Using the same volume computation as in the compact case, the volume of the finite sub-tube is
\begin{equation}
\vert T_L\left
(P,a\right)\vert=L\omega_{n-2}\int_0^a \left(S\left(r\right)\right)^{n-2}C\left(r\right)\operatorname{dr},
\end{equation}
where \(\omega_{n-2}=\vert\mathbb{S}^{n-2}\left(1\right)\vert\). Since
\begin{equation}
\frac{d}{dr}S\left(r\right)^{n-1}=\left(n-1\right)S\left(r\right)^{n-2}S^{\prime}\left(r\right)=\left(n-1\right)S\left(r\right)^{n-2}C\left(r\right),
\end{equation}
we have
\begin{equation}
\vert T_L\left
(P,a\right)\vert=\frac{L\omega_{n-2}\left(S\left(a\right)\right)^{n-1}}{n-1}.
\end{equation}
Then the lateral boundary area is
\begin{equation}
\vert\partial T_{lat}\left(P,a\right)\vert=L\omega_{n-2}\left(S\left(a\right)\right)^{n-2}C\left(a\right).
\end{equation}
The two end caps contribute with
\begin{equation}
\vert\partial_{cap}  T_L\left(P,a\right)\vert= 2\omega_{n-2}\int_0^a\left(S\left(r\right)\right)^{n-2}\operatorname{dr}.
\end{equation}
Therefore, the full boundary area of \(T_L\left(P,a\right)\) is given by
\begin{equation}
\vert\partial T_L\left(P,a\right)\vert=L\omega_{n-2}\left(S\left(a\right)\right)^{n-2}C\left(a\right)+2\omega_{n-2}\int_0^a\left
(S\left(r\right)\right)^{n-2}\operatorname{dr}.
\end{equation}
Dividing by the volume gives us
\begin{equation}
\frac{\vert\partial T_L\left(P,a\right)\vert}{\vert T_L\left(P,a\right)\vert}=\left(n-1\right)\frac{C\left(a\right)}{S\left(a\right)}+\frac{2\omega_{n-2}\int_0^a\left(S\left(r\right)\right)^{n-2}\operatorname{dr}}{\frac{L\omega_{n-2}\left(S\left(a\right)\right)^{n-1}}{n-1}}
\end{equation}

The second term is clearly positive for every finite \(L\) and it converges to zero as \(L\) goes to infinity. Hence,
\begin{equation}
\lim_{L\longrightarrow\infty}\frac{\vert\partial T_L\left(P,a\right)\vert}{\vert T_L\left(P,a\right)\vert}=\left(n-1\right)\frac{C\left(a\right)}{S\left(a\right)}.
\end{equation}
Consequently, the Cheeger constant of the unbounded tube satisfies the upper estimate
\begin{equation}\label{eq.158}
h\left(T_{\infty}\left(P,a\right)\right)\leq\left(n-1\right)\frac{C\left(a\right)}{S\left(a\right)}.
\end{equation}

Now comes the more delicate part regarding the lower bound of \(h\left(T_{\infty}\left(P,a\right)\right)\). The main issue is that the global moving-sphere parametrization does not automatically follow from the simple embeddedness of the unbounded curved tube \(T_{\infty}\left(P,a\right)\). In particular, the inequality \(a<\rho_{tub}\left(P\right)\) gives the uniqueness of the nearest-point/Fermi representation \(x=\operatorname{exp}_p\left(r\zeta\right)\), with \(0<r<a\), \(\zeta\in \nu_pP\) and \(\vert\zeta
\vert=1\). However, the moving-sphere representation is different: \(x=\operatorname{exp}_q\left(a\xi\right)\), \(\langle\xi, \tau_q\rangle>0\), with \(\vert\xi\vert=1\) and \(q\in P\). Here, the center \(q\in P\) is not the nearest point to \(x\). It is a point on the curve \(P\) such that \(x\) lies on the geodesic sphere of fixed radius \(a\) centered at \(q\). Looking back at the proof of proposition \ref{prop.4}, we have used the fact that the sublevel set \(A_x=\{s\in\mathbb{R}\hspace{0.1cm}\vert\hspace{0.1cm}d\left(x,\gamma\left(s\right)\right)<a\}\) is a connected interval on the oriented compact curve. This implies that it has exactly one entering boundary point. That provides the uniqueness of the positively oriented representation. For a noncompact \(P\), we can no longer control the global structure of \(A_x\) through compactness. However, properness of the embedding of \(P\) in \(X^n\), together with a positive uniform tubular radius, provides exactly the missing global control. Under these assumptions, the argument of proposition \ref{prop.4} naturally extends to the noncompact setting.

This leads us to the following proposition.

\begin{prop}\label{prop.5}
Let \(X^n\) be a real space form with constant sectional curvature \(K^X\leq 0\) and let \(T_{\infty}\left(P,a\right)=\{x\in X\hspace{0.1cm}\vert\hspace{0.1cm}\operatorname{dist}\left(x,P\right)\leq a\}\) be the unbounded curved tube of radius \(a>0\) around a smooth properly embedded noncompact curve \(P\). Assume that \(P\) possesses a positive uniform tubular radius \(\rho_{tub}\left(P\right)\). Specifically, suppose that there exist \(0<a<\rho\) such that the normal exponential map
\begin{equation}
\exp^{\perp}:\{\left
(s,v\right)\hspace{0.1cm}\vert\hspace{0.1cm}s\in\mathbb{R},\hspace{0.1cm}v\in\nu_{\gamma\left(s\right)}P,\hspace{0.1cm}\vert v\vert<\rho\}\longrightarrow\operatorname{Int}\left(T_{\infty}\left(P,\rho\right)\right)
\end{equation}
is a smooth diffeomorphism. Then the moving-sphere parametrization 
\begin{equation}
\Psi_a:\mathcal{H}_+\longrightarrow\operatorname{Int}\left(T_{\infty}\left(P,a\right)\right), \quad \Psi_a\left(s,\xi\right)=\operatorname{exp}_{\gamma\left(s\right)}\left(a\xi\right),
\end{equation}
where \(\mathcal{H}_+:=\{\left(s,\xi\right)\hspace{0.1cm}\vert\hspace{0.1cm} s\in\mathbb{R},\xi\in T_{\gamma\left(s\right)}X,\vert\xi\vert=1,\langle\xi, T\left(s\right)\rangle>0\}\),
is a smooth diffeomorphism. 
\end{prop}

\newpage
\hspace{-0.55cm}\textbf{Proof of proposition \ref{prop.5}.} 

The argument remains identical to that in proposition \ref{prop.4}. Nonetheless, if we want to be precise, it is appropriate to verify that the relevant sublevel set remains a single bounded interval. 

Let us fix \(x\in\operatorname{Int}\left(T_{\infty}\left(P,a\right)\right)\) and define \(F_x\left(s\right):=d\left(x,\gamma\left(s\right)\right)^2\), together with the set \(A_x=\{s\in\mathbb{R}\hspace{0.1cm}\vert\hspace{0.1cm}F_x\left(s\right)<a^2\}\). Since \(x\in\operatorname{Int}\left
(T_{\infty}\left(P,a\right)\right)\), the set \(A_x\) is evidently nonempty. Furthermore, because \(X^n\) is complete, the closed ball \(\overline{B}_a\left(x\right)\) is compact. The arc-length parametrization \(\gamma:\mathbb{R}\longrightarrow X^n\) is proper, so \(\gamma^{-1}\left(\overline{B}_a\left(x\right)\right)\) is compact in \(\mathbb{R}\). This implies that \(A_x\) is also bounded.

Now, presume that \(A_x\) possesses two different connected components \(I_1\) and \(I_2\). On each component, the previously defined continuous function \(F_x\) satisfies \(F_x<a^2\). In particular, at the boundary of the component, we get \(F_x=a^2\). Provided \(I_j\) are bounded for \(j\in\{1,2\}\), the function \(F_x\) attains an interior minimum at some \(s_j\in I_j\). At this point, we have \(F_x^{\prime}\left(s_j\right)=0\). Since \(d\left(x,\gamma\left(s_j\right)\right)<a\), the minimizing geodesic from \(\gamma\left(s_j\right)\) to \(x\) is unique by the preceding radius assumptions. Hence, the first variation formula gives us
\begin{equation}
F_x^{\prime}\left(s_j\right)=-2\langle \operatorname{exp}^{-1}_{\gamma\left(s_j\right)}\left(x\right), T\left(s_j\right)\rangle=0,
\end{equation}
where \(T\left(s_j\right):=\dot{\gamma}\left(s_j\right)\) is the positively oriented unit tangent vector at \(s_j\). Consequently,
\begin{equation}
v_j:=\operatorname{exp}^{-1}_{\gamma\left(s_j\right)}\left(x\right)\perp T_{\gamma\left(s_j\right)}P\implies v_j\in\nu_{\gamma\left(s_j\right)} P,\quad \vert v_j\vert<a,
\end{equation}
so
\begin{equation}
x=\operatorname{exp}_{\gamma\left(s_j\right)}\left(v_j\right).
\end{equation}
This implies that the two distinct components \(I_1,I_2\) provide two different normal representations
\begin{equation}
x=\operatorname{exp}_{\gamma\left(s_1\right)}\left(v_1\right)=\operatorname{exp}_{\gamma\left(s_2\right)}\left(v_2\right),\quad s_1\neq s_2,
\end{equation}
contradicting the assumed injectivity of the normal exponential map. 

Therefore, \(A_x\) has at most one connected component, and since it is nonempty, \(A_x\) is a single open interval. This is precisely the mechanism already presented in proposition \ref{prop.4}. Here we justify this point rigorously for noncompact \(P\).

Finally, we shall make a short comment on the transversality property of \(A_x\). Here, \(A_x=\left(s_-\left(x\right),s_+\left(x\right)\right)\), where \(s_-\left(x\right),s_+\left
(x\right)\) are unique endpoints such that \(s_-\left(x\right)<s_+\left
(x\right)\). 

Exactly as in proposition \ref{prop.4}, the boundary intersections remain transverse. Indeed, if \(F^{\prime}_x\left(s_*\right)=0\) at some \(s_*\in\{s_-\left
(x\right),s_+\left(x\right)\}\), then the first variation formula implies that the minimizing geodesic of length \(a\) from \(\gamma\left(s_*\right)\) to \(x\) is normal to \(P\). Hence \(x=\operatorname{exp}^{\perp}_{\gamma\left(s_*\right)}\left(a\xi_*\right)\), for some unit normal vector \(\xi_*\). However, \(x\in\operatorname{Int}\left(T_{\infty}\left(P,a\right)\right)\) also has its unique Fermi representation \(x=\operatorname{exp}^{\perp}_{\gamma\left(s_0\right)}\left(v_0\right)\), \(\vert v_0\vert<a\), with \(s_0\in A_x\) whereas \(s_*\in\partial A_x\). Since \(a<\rho\), both representations lie in the injectivity domain of \(\operatorname{exp}^{\perp}\), yielding a contradiction. Therefore, 
\begin{equation}
F^{\prime}_x\left(s_{\pm}\left(x\right)\right)\neq 0.
\end{equation}
Since \(s_-\left(x\right), s_+\left(x\right)\) are respectively the entering and exiting endpoints of \(A_x\), we have
\begin{equation}
F^{\prime}_x\left(s_-\left(x\right)\right)<0,\quad F^{\prime}_x\left(s_+\left(x\right)\right)>0.
\end{equation}

In summary, the unique entering endpoint of \(A_x\) determines the unique positively oriented moving-sphere representation of \(x\), while the smooth dependence on \(x\) follows from the implicit function theorem exactly as in the compact case mentioned above. Therefore, the moving-sphere parametrization is again a smooth diffeomorphism even in the noncompact case.

This completes the proof of proposition \ref{prop.5}.

\vspace{0.2cm}
Using the diffeomorphism of \(\Psi_a\), every point \(x\in \operatorname{Int}\left(T_{\infty}\left(P,a\right)\right)\) has a unique representation
\begin{equation}
x=\operatorname{exp}_{\gamma\left(s_-\left(x\right)\right)}\left(a\xi_-\left(x\right)\right), \quad \vert\xi_-\left(x\right)\vert=1, \quad \langle\xi_-\left(x\right), T\left(s_-\left(x\right)\right)\rangle>0.
\end{equation}
Let us define the following vector field
\begin{equation}
V\left(x\right)=\restr{\frac{d}{dr}}{r=a}\operatorname{exp}_{\gamma\left(s_-\left(x\right)\right)}\left(r\xi_-\left(x\right)\right).
\end{equation}
Equivalently, \(V\left(x\right)\) is the outward radial unit vector evaluated at \(x\) of the geodesic sphere of radius \(a\) centered at \(\gamma\left(s_-\left(x\right)\right)\). Because \(s_-\left(x\right)\) and \(\xi_-\left(x\right)\) are smooth, the vector field \(V\) is also smooth on \(\operatorname{Int}\left(T_{\infty}\left(P,a\right)\right)\).

One can see that this is exactly the type of global vector field used in the compact lower-bound argument. Therefore, the divergence computation remains identical, and we would eventually get
\begin{equation}\label{eq.194}
\operatorname{div}\left(V\right)=\left(n-1\right)\frac{C\left(a\right)}{S\left(a\right)}.
\end{equation}

Now, let \(N\subset T_{\infty}\left(P,a\right)\) be a smooth admissible competitor with compact closure and positive finite volume \(\vert N\vert\). Let us also denote the divergence (\ref{eq.194}) as \(\operatorname{div}\left(V\right)=c\). Then, using the divergence approach (\ref{eq:43}), we obtain 
\begin{equation}
\frac{\vert\partial N\vert}{\vert N\vert}\geq c.
\end{equation}
Taking the infimum over all admissible \(N\), we get
\begin{equation}\label{eq.196}
h\left(T_{\infty}\left(P,a\right)\right)\geq\left(n-1\right)\frac{C\left(a\right)}{S\left(a\right)}.
\end{equation}

The combination of the upper bound (\ref{eq.158}) and the lower bound (\ref{eq.196}) yields
\begin{equation}\label{eq.197}
h\left(T_{\infty}\left(P,a\right)\right)=\left(n-1\right)\frac{C\left(a\right)}{S\left(a\right)}.
\end{equation}

We proved that the compact upper/lower-bound arguments can be extended to the unbounded setup, providing the equality (\ref{eq.197}). First of all, the exhaustion is given by longer and longer finite sub-tubes \(T_L\left(P,a\right)\). The reason is simple, both volume and lateral boundary area grow linearly in the variable \(L\), while the two end caps remain uniformly bounded. Hence, the contribution of the end caps disappears in the limit \(L\longrightarrow\infty\). Considering the lower-bound, the global validity of the calibration vector field is not automatic. We had to provide proposition \ref{prop.5}, in order to justify that the moving-sphere parametrization is globally well-defined on \(T_{\infty}\left(P,a\right)\) and that the associated calibration vector field can be applied to every compact admissible competitor. With such a global assumption, we managed to extend the compact lower-bound argument to the unbounded case.  

Although the Cheeger constant of the unbounded curved tube coincides with the compact case under the global moving-sphere calibration assumption (see proposition \ref{prop.5}), this value is not attained by any finite-volume admissible set. For instance, let \(E \subset T_{\infty}\left(P,a\right)\) be a smooth admissible competitor with compact closure and positive volume. Then by the divergence theorem, we have
\begin{equation}
\int_E\operatorname{div}\left(V\right)=\int_{\partial E}\langle V,n_E\rangle,
\end{equation}
where \(n_E\) is the outward unit normal to \(\partial E\). From (\ref{eq.194}), we know that the divergence is constant, i.e.  \(\operatorname{div}\left(V\right)=c\), so
\begin{equation}
c\vert E\vert=\int_{\partial E}\langle V,n_E\rangle.
\end{equation}
Since \(\vert V\vert=1\) and \(\vert n_E\vert=1\), we have \(\langle V, n_E\rangle\leq 1\) (pointwise). Therefore,
\begin{equation}
c\vert E\vert=\int_{\partial E}\langle V,n_E\rangle\leq\vert\partial E\vert\implies\frac{\vert\partial E\vert}{\vert E\vert}\geq c.
\end{equation}
Now, suppose that \(E\) is a Cheeger set. Then
\begin{equation}
\frac{\vert\partial E\vert}{\vert E\vert}=h\left(T_{\infty}\left(P,a\right)\right)=c.
\end{equation}
Thus
\begin{equation}
\vert\partial E\vert=c\vert E\vert.
\end{equation}
But from the calibration inequality, we also have
\begin{equation}
c\vert E\vert=\int_{\partial E}\langle V, n_E\rangle\leq\vert\partial E\vert=c\vert E\vert.
\end{equation}

Therefore, this equality must hold when \(\langle V,n_E\rangle\leq 1\). Since both \(V\) and \(n_E\) are unit vectors, we obtain an equality
\begin{equation}
\langle V, n_E\rangle=1\implies
n_E=V
\end{equation}
on \(\partial E\). This points to the fact that any finite-volume Cheeger set must satisfy this very rigid boundary condition: the outward unit normal of \(\partial E\) must equal the calibration vector field \(V\). 

Now, we shall explore the consequences of this restrictive condition. Let 
\begin{equation}
\Psi^{-1}_a\left(x\right)=\left(\sigma\left(x\right),\xi\left(x\right)\right)
\end{equation}
be the global moving-sphere inverse. Thus
\begin{equation}
x=\operatorname{exp}_{\gamma\left(\sigma\left(x\right)\right)}\left(a\xi\left(x\right)\right), \quad \vert\xi\left(x\right)\vert=1,\quad\langle\xi\left(x\right),T\left(\sigma\left(x\right)\right)\rangle>0.
\end{equation}
Here, the function 
\begin{equation}
\sigma:\operatorname{Int}\left(T_{\infty}\left(P,a\right)\right)\longrightarrow\mathbb{R}
\end{equation}
records the parameter value of the positively oriented moving-sphere center. 

Now, we shall compute \(d\sigma\left(V\right)\) and focus on its sign. Let us define
\begin{equation}
G\left(s,x\right):=d\left(x,\gamma\left(s\right)\right)^2-a^2.
\end{equation}
By definition of \(\sigma\left(x\right)\), we have
\begin{equation}
G\left(\sigma\left(x\right),x\right)=0.
\end{equation}
Differentiating this identity in the direction of the calibration vector field \(V\), we get
\begin{equation}\label{eq.210}
\partial_sG\left(\sigma\left(x\right),x\right)d\sigma\left(V\right)+d_xG\left(\sigma\left(x\right),x\right)[V]=0.
\end{equation}
Here, we compute the two terms. The first one is obtained by using the first variational formula,
\begin{equation}
\partial_sG\left(\sigma\left(x\right),x\right)=-2a\langle\xi\left(x\right),T\left(\sigma\left(x\right)\right)\rangle.
\end{equation}
Considering the second term, since \(x\) lies at a distance \(a\) from \(\gamma\left(\sigma\left(x\right)\right)\) in the direction \(\xi\left(x\right)\), the \(x\)-gradient of \(d\left(x,\gamma\left(\sigma\left(x\right)\right)\right)^2\) is \(2aV\). Therefore,
\begin{equation}
d_xG\left(\sigma\left(x\right),x\right)[V]=\langle 2aV,V\rangle=2a.
\end{equation}
Substituting this back into the differentiated identity (\ref{eq.210}) gives
\begin{equation}
-2a\langle\xi\left(x\right),T\left(\sigma\left(x\right)\right)\rangle d\sigma\left(V\right)+2a=0.
\end{equation}
Hence
\begin{equation}\label{eq.215}
d\sigma\left(V\right)=\frac{1}{\langle\xi\left(x\right),T\left(\sigma\left(x\right)\right)\rangle}.
\end{equation}
Since the moving-sphere parametrization is positively oriented, the denominator of (\ref{eq.215}) is strictly positive. 

Therefore,
\begin{equation}
d\sigma\left(V\right)>0
\end{equation}
everywhere in \(\operatorname{Int}\left(T_{\infty}\left(P,a\right)\right)\).
This shows that the calibration vector field \(V\) strictly increases the center coordinate \(\sigma\).

\newpage
In summary, the equality in the calibration argument would force the outward normal of the boundary of a minimizer to coincide with the vector field \(V\). Since the associated center coordinate \(\sigma\) satisfies \(d\sigma\left(V\right)>0\), no relatively compact smooth admissible set can satisfy this condition. Since \(\overline{E}\) is compact, \(\sigma\) attains its minimum on \(\overline{E}\). The inequality \(d\sigma\left(V\right)>0\) excludes an interior minimum. Thus, a minimum is attained at some point on the boundary \(\partial E\). At such a boundary minimum, the outward normal derivative satisfies \(d\sigma\left(n_E\right)\leq 0\). However, equality in the calibration requires \(n_E=V\), while \(d\sigma\left(V\right)>0\), so we obtain a contradiction. 

Hence, the infimum is approached by an exhaustion of the unbounded tube by longer and longer finite sub-tubes, but no finite-volume admissible set attains it, i.e., there is no Cheeger set in the admissible class. This actually implies that the Cheeger constant \(h\left(T_{\infty}\left(P,a\right)\right)\) exists as a finite number but is not attained by a finite-volume set. In particular, every finite truncation (sub-tube) \(T_L\left(P,a\right)\) has an extra positive boundary contribution coming from its two end caps. Thus, no finite truncation realizes the limiting quotient exactly. Instead, the minimizing sequence is given by longer and longer sub-tubes \(T_L\left(P,a\right)\), and this sequence escapes to infinity along the underlying curve \(P\). Therefore, the formal limiting object is the whole unbounded tube \(T_{\infty}\left(P,a\right)\).

\newpage
\section*{Acknowledgements}
I would like to acknowledge several people for their valuable advice, help, and guidance. First, I would like to thank Prof. David Krejčiřík, DSc., who introduced me to this topic, supervised our joint paper \cite{krejcirik2019cheeger}, and originally formulated the conjecture corresponding to theorem \ref{thm:1} in dimension \(n=3\). I would also like to thank Prof. Jan Slovák, DrSc., for his considerable assistance and valuable advice during the preparation and completion of this paper.

Besides, I would like to acknowledge the financial support received from several sources. This research was supported by the Czech Science Foundation project Cartan Supergeometries and Higher Cartan Geometries, No. 24-10887S. Further support was provided by the European Union through Horizon Europe project No. 101086123, which enabled a research stay at UC Davis. Additionally, there was also a local support from the project for Specific research in Mathematics MUNI/A/1457/2023 for doctoral students, provided by the Department of Mathematics and Statistics at Masaryk University.  

\vspace{0.3cm}
\section*{Data availability}
No data were used for the research described in this article.

\section*{Declaration of competing interest}
The author declares no competing interests.

\bibliographystyle{elsarticle-num}

\bibliography{cas-refs}

@book{maggi2012sets,
  title={Sets of finite perimeter and geometric variational problems: an introduction to Geometric Measure Theory},
  author={Maggi, Francesco},
  number={135},
  year={2012},
  publisher={Cambridge University Press}
}

@incollection{leonardi2015overview,
  author    = {Leonardi, Gian Paolo},
  title     = {An overview on the Cheeger problem},
  booktitle = {New Trends in Shape Optimization},
  series    = {International Series of Numerical Mathematics},
  volume    = {166},
  pages     = {117--139},
  publisher = {Birkh{\"a}user/Springer},
  year      = {2015}
}

@inproceedings{singmaster1978constrained,
  title={A constrained isoperimetric problem},
  author={Singmaster, David and Souppouris, DJ},
  booktitle={Mathematical Proceedings of the Cambridge Philosophical Society},
  volume={83},
  number={1},
  pages={73--82},
  year={1978},
  organization={Cambridge University Press}
}

@incollection{cheeger1970lower,
  author    = {Cheeger, Jeff},
  title     = {A lower bound for the smallest eigenvalue of the Laplacian},
  booktitle = {Problems in Analysis: A Symposium in Honor of Salomon Bochner},
  editor    = {Gunning, Robert C.},
  pages     = {195--199},
  publisher = {Princeton University Press},
  year      = {1970}
}

@article{parini2011introduction,
  title={An introduction to the Cheeger problem},
  author={Parini, Enea},
  journal={Surv. Math. Appl.},
  volume={6},
  pages={9--21},
  year={2011}
}

@article{kawohl2003isoperimetric,
  title={Isoperimetric estimates for the first eigenvalue of the $ p $-Laplace operator and the Cheeger constant},
  author={Kawohl, Bernd and Fridman, Vladislav},
  journal={Commentationes Mathematicae Universitatis Carolinae},
  volume={44},
  number={4},
  pages={659--667},
  year={2003},
  publisher={Charles University in Prague, Faculty of Mathematics and Physics}
}

@article{kawohl2006characterization,
  title={Characterization of Cheeger sets for convex subsets of the plane},
  author={Kawohl, Bernd and Lachand-Robert, Thomas},
  journal={Pacific journal of mathematics},
  volume={225},
  number={1},
  pages={103--118},
  year={2006},
  publisher={Mathematical Sciences Publishers}
}

@article{canete2022cheeger,
  title={Cheeger sets for rotationally symmetric planar convex bodies},
  author={Canete, Antonio},
  journal={Results in Mathematics},
  volume={77},
  number={1},
  pages={9},
  year={2022},
  publisher={Springer}
}

@article{krejcirik2011cheeger,
  author  = {Krej{\v{c}}i{\v{r}}{\'\i}k, David and Pratelli, Aldo},
  title   = {The Cheeger constant of curved strips},
  journal = {Pac. J. Math.},
  volume  = {254},
  number  = {2},
  pages   = {309--333},
  year    = {2011},
  doi     = {10.2140/pjm.2011.254.309}
}

@article{krejcirik2019cheeger,
  author  = {Krej{\v{c}}i{\v{r}}{\'\i}k, David
             and Leonardi, Gian Paolo
             and Vlachopulos, Petr},
  title   = {The Cheeger constant of curved tubes},
  journal = {Arch. Math.},
  volume  = {112},
  pages   = {429--436},
  year    = {2019},
  doi     = {10.1007/s00013-018-1282-x}
}

@book{gray2003tubes,
  title={Tubes},
  author={Gray, Alfred},
  volume={221},
  year={2003},
  publisher={Springer Science \& Business Media}
}

@book{evans2018measure,
  title={Measure theory and fine properties of functions},
  author={Evans, LawrenceCraig},
  year={2018},
  publisher={Routledge}
}

@book{petersen2006riemannian,
  title={Riemannian geometry},
  author={Petersen, Peter},
  volume={171},
  year={2006},
  publisher={Springer}
}


\end{document}